\documentclass[11pt,a4paper,reqno]{amsart}

\usepackage{stmaryrd}
\usepackage{amssymb}
\usepackage{amsmath}
\usepackage{amsfonts}
\usepackage{amsthm}
\usepackage{color}
\usepackage[usenames,dvipsnames]{xcolor} 
\usepackage{bbm}
\usepackage{graphicx}
\usepackage{rotating}
\usepackage{hyperref}
\usepackage{mathrsfs}
\usepackage{cmll}
\usepackage[all, cmtip]{xy}
\usepackage{amscd}
\usepackage{todonotes}
\providecommand{\U}[1]{\protect\rule{.1in}{.1in}}
\newtheorem{theorem}{Theorem}[subsection]

\newtheorem{lemma}[theorem]{Lemma}
\newtheorem{proposition}[theorem]{Proposition}
\newtheorem*{theorem*}{Theorem}

\theoremstyle{definition}
\newtheorem{definition}[theorem]{Definition}
\newtheorem{examplex}[theorem]{Example}

\theoremstyle{remark}
\newtheorem{remarkx}[theorem]{Remark}
\newenvironment{remark}
  {\pushQED{\qed}\remarkx}
  {\popQED\endremarkx}

\thanks{}

\numberwithin{equation}{section}

\makeatletter
\def\@tocline#1#2#3#4#5#6#7{\relax
  \ifnum #1>\c@tocdepth 
  \else
    \par \addpenalty\@secpenalty\addvspace{#2}%
    \begingroup \hyphenpenalty\@M
    \@ifempty{#4}{%
      \@tempdima\csname r@tocindent\number#1\endcsname\relax
    }{%
      \@tempdima#4\relax
    }%
    \parindent\z@ \leftskip#3\relax \advance\leftskip\@tempdima\relax
    \rightskip\@pnumwidth plus4em \parfillskip-\@pnumwidth
    #5\leavevmode\hskip-\@tempdima
      \ifcase #1
       \or\or \hskip 1.5 em \or \hskip 2em \else \hskip 3em \fi%
      #6\nobreak\relax
    \hfill\hbox to\@pnumwidth{\@tocpagenum{#7}}\par
    \nobreak
    \endgroup
  \fi}
\makeatother

\usepackage[utf8x]{inputenc}

\begin{document}
\newcommand{\R}{{\mathbbm R}}
\newcommand{\C}{{\mathbbm C}} 
\newcommand{\T}{{\mathbbm T}}
\newcommand{\D}{{\mathbbm D}}
\renewcommand{\P}{\mathbb P}

\newcommand{\Aa}{{\mathcal A}}
\newcommand{\Ii}{{\mathbbm K}}
\newcommand{\Jj}{{\mathbbm J}}
\newcommand{\Nn}{{\mathcal N}}
\newcommand{\Ll}{{\mathcal L}}
\newcommand{\Tt}{{\mathcal T}}
\newcommand{\Gg}{{\mathcal G}}
\newcommand{\Dd}{{\mathcal D}}
\newcommand{\Cc}{{\mathcal C}}
\newcommand{\Oo}{{\mathcal O}}

\newcommand{\pr}{\operatorname{pr}}
\newcommand{\bla}{\langle \hspace{-2.7pt} \langle}
\newcommand{\bra}{\rangle\hspace{-2.7pt} \rangle}
\newcommand{\blq}{[ \! [}
\newcommand{\brq}{] \! ]}
 \newcommand{\into}{\mathbin{\vrule width1.5ex height.4pt\vrule height1.5ex}}

\newcommand{\TODO}[1]{
     \todo[
     size=\tiny]{#1}
     }

\setlength{\marginparwidth}{2cm}

\title{Deformations of Linear Lie Brackets}

\author{Pier Paolo La Pastina}
\address{Dipartimento di Matematica ``Guido Castelnuovo'', Universit\`a degli studi di Roma ``La Sapienza'', P.le Aldo Moro 5, I-00185 Roma, Italy.}
\email{lapastina@mat.uniroma1.it}

\author{Luca Vitagliano}
\address{DipMat, Universit\`a degli Studi di Salerno, via Giovanni Paolo II n${}^\circ$ 123, 84084 Fisciano (SA) Italy.}
\email{lvitagliano@unisa.it}

\begin{abstract}
A VB-algebroid is a vector bundle object in the category of Lie algebroids. We attach to every VB-algebroid a differential graded Lie algebra and we show that it controls deformations of the VB-algebroid structure. Several examples and applications are discussed. This is the first in a series of papers devoted to deformations of vector bundles and related structures over differentiable stacks.
\end{abstract}

\maketitle

\tableofcontents

\section*{Introduction}

Lie algebroids are ubiquitous in differential geometry: they encompass several algebraic and geometric structures such as Lie algebras, tangent bundles, foliations, Poisson brackets, Lie algebra actions on manifolds and so on, and they are the infinitesimal counterparts of Lie groupoids. The notion of Lie algebroid appeared for the first time in the work of Pradines \cite{pradines:theorie} and has become more and more important in the last 50 years. In particular, deformations of Lie  {algebroids have been discussed by Crainic and Moerdijk \cite{crainic:def}, while deformations of Lie groupoids have been studied very recently by Crainic, Mestre and Struchiner \cite{crainic:def2}}.

VB-algebroids are vector bundle objects in the category of Lie algebroids \cite{mackenzie:double, gracia-saz:vb}. They emerge naturally in the study of Lie algebroids. For instance, the tangent and the cotangent bundles of a Lie algebroid are VB-algebroids. Additionally, VB-algebroids are generalizations of ordinary representations of Lie algebroids: specifically they are equivalent to 2-term representations up to homotopy of Lie algebroids, hence to (special kinds of) representations of Lie algebroids on graded vector bundles \cite{abad:ruth, gracia-saz:vb}. Finally, VB-algebroids are the infinitesimal counterparts of VB-groupoids. The latter serve as models for vector bundles over certain singular spaces: \emph{differentiable stacks} \cite{behrend:diff}. Examples of differentiable stacks are orbifolds, leaf spaces of foliations and orbit spaces of Lie group actions.

This is the first in a series of papers devoted to deformations of vector bundles over differentiable stacks and related deformation problems. A first step in this direction has been taken by del Hoyo and Ortiz in \cite{delhoyo:morita}, where the authors show that the VB-cohomology of a VB-groupoid is actually \emph{VB-Morita invariant}, i.e.~it is an invariant of the associated vector bundle of differentiable stacks. Notice that several important geometric structures, like Riemannian metrics, symplectic forms, complex structures, etc., can be seen as vector bundle maps. In order to study deformations of the former, it is then useful to study deformations of vector bundles themselves first. In this paper, we begin this program working at the infinitesimal level, i.e.~studying deformations of VB-algebroids. More precisely, we study deformations of VB-algebroid structures on double vector bundles. In the second paper of the series we will study deformations of VB-groupoids and their behavior under the Lie functor \cite{lapastina:def2}.

The paper is divided in two main sections. The first one presents the general theory, and the second one discusses examples and applications. In its turn, the first section is divided in four subsections. In Subsection \ref{Sec:deformations1} we recall from \cite{crainic:def} the differential graded Lie algebra (DGLA) controlling deformations of Lie algebroids.  We also discuss gauge equivalent deformations, something that is missing in the original discussion by Crainic and Moerdijk. In Subsection \ref{Sec:deformations2} we recall the basics of VB-algebroids, in particular their description in terms of graded manifolds 
In Subsection \ref{sec:lin_def} we discuss deformations of VB-algebroids. Let $(W \Rightarrow E; A \Rightarrow M)$ be a VB-algebroid. In particular $W \Rightarrow E$ is a Lie algebroid, so it has an associated deformation DGLA. We show that deformations of $(W \Rightarrow E; A \Rightarrow M)$ are controlled by the sub-DGLA of  \emph{linear} cochains, originally introduced in \cite{etv:infinitesimal}, that we call the \emph{linear deformation complex}, and we provide various equivalent descriptions of this object. The most efficient one involves the \emph{homogeneity structure} of $W \to A$, i.e. the action of the monoid $\mathbb R_{\geq 0}$ on the total space by fiber-wise homotheties: linear deformation cochains are precisely those that are invariant under the action by (non-zero) homotheties. It is clear that this action induces graded subalgebras of the algebras of functions, differential forms and multivectors on the total space of a vector bundle and can be used to define \emph{linear} objects in these algebras, thus giving a unified framework to the original definitions in \cite{bursztyn:mult} and \cite{iglesias:univ}. \color{black} We recall this briefly in Appendix \ref{Sec:homogeneity}. Another important description of the linear deformation complex is in terms of graded geometry. It is well-known that Lie algebroids are equivalent to DG-manifolds concentrated in degree 0 and 1 and VB-algebroids are equivalent to vector bundles in the category of such graded manifolds \cite{mehta:thesis, vaintrob:lie, voronov:q}. Moreover, it is (implicitly) shown already in \cite{crainic:def} that the deformation DGLA of a Lie algebroid $A \Rightarrow M$ is isomorphic to the DGLA of vector fields on $A[1]$, giving an elegant and manageable interpretation. A similar interpretation becomes very useful in the case of VB-algebroids, particularly in some examples.

In Subsection \ref{sec:linearization_1} we show that it is possible to ``linearize'' deformation cochains of the top algebroid $W \Rightarrow E$ of a VB-algebroid $(W \Rightarrow E; A \Rightarrow M)$, adapting a technique from \cite{cabrera:hom}. The main consequence is that the linear deformation cohomology is embedded, as a graded Lie algebra, in the deformation cohomology of the top algebroid.

In the second section of the paper we present examples. We discuss in details particularly simple instances of VB-algebroids coming from linear algebra, namely VB-algebras and LA-vector spaces (Subsections \ref{sec:VB-alg} and \ref{sec:LA-vect} respectively). VB-algebras are equivalent to Lie algebra representations, and our discussion encompasses the classical theory of Nijenhuis and Richardson \cite{nijenhuis:coh, nijenhuis:def}. In Subsection \ref{sec:tangent-VB}, we discuss deformations of the tangent and the cotangent VB-algebroids of a Lie algebroid. Partial connections along foliations and Lie algebra actions on vector bundles can be also encoded by VB-algebroids and we study the associated deformation complexes in Subsections \ref{sec:partial_conn} and \ref{sec:Lie_vect} respectively. We also discuss VB-algebroids of type $1$ in the classification of Gracia-Saz and Mehta \cite{gracia-saz:vb}. Their deformation cohomology is canonically isomorphic to that of the base algebroid (Subsection \ref{sec:type_1}). 

We usually indicate with a bullet ${}^\bullet$ the presence of a degree in a graded vector space. If $V^\bullet$ is a graded vector space, its \emph{shift by one} $V[1]^\bullet$ is defined by $V[1]^k = V^{k+1}$. We assume the reader is familiar with graded manifolds and the graded geometry description of Lie algebroids.
Here, we only recall that a graded manifold is \emph{concentrated} in degree $k, \ldots, k+l$, if the degrees of its coordinates range from $k$ to  $k+l$ and a DG-manifold is a graded manifold equipped with an homological vector field. For instance, if $A \Rightarrow M$ is a Lie algebroid, then shifting by one the degree of the fibers of the vector bundle $A \to M$, we get a DG-manifold $A[1]$, concentrated in degree $0$ and $1$, whose homological vector field is the de Rham differential $d_A$ of $A$. Explicitly, the algebra of smooth functions on $A[1]$ is $C^\infty (A[1]) = \Omega^\bullet_A := \Gamma (\wedge^\bullet A^*)$. Correspondence $A \rightsquigarrow A[1]$ establishes an equivalence between the category of Lie algebroids and the category of DG-manifolds concentrated in degree $0$ and $1$ \cite{vaintrob:lie}. We stress that the graded manifold $A[1]$ is obtained from $A$ by \emph{assigning a degree $1$} to the linear fiber coordinates. We warn the unfamiliar reader that, despite the notation, the shift $A \rightsquigarrow A[1]$ is (related but) different from the degree shift for a graded vector space discussed at the beginning of this paragraph. The reader can find more details \cite{mehta:thesis} which is also our main reference for graded geometry.

\section{Deformations of VB-algebroids}

\subsection{Deformations of Lie algebroids}\label{Sec:deformations1}

A Lie algebroid $A \Rightarrow M$ over a manifold $M$ is a vector bundle $A \to M$ with a Lie bracket $[-,-]$ on its space of sections $\Gamma (A)$ and a bundle map $\rho: A \to TM$, satisfying the Leibniz rule:
\[
[\alpha, f \beta] = \rho(\alpha)(f) \beta + f [\alpha, \beta],
\]
for all $\alpha, \beta \in \Gamma(A)$, $f \in C^{\infty}(M)$.

We briefly recall the deformation theory of Lie algebroids, originally due to Crainic and Moerdijk \cite{crainic:def}, adding some small details about equivalence of deformations which are missing in the original treatment. We begin with a vector bundle $E \to M$. Let $k \geq 0$.

\begin{definition} A  \emph{multiderivation} of $E$ with $k$ entries (and $C^{\infty}(M)$-multilinear symbol), also called a $k$-\emph{derivation}, is a skew-symmetric, $\mathbb{R}$-$k$-linear map
\[
c: \Gamma(E) \times \dots \times \Gamma(E) \to \Gamma(E)
\]
such that there exists a bundle map $\sigma_c: \wedge^{k-1} E \to TM$, the \emph{symbol} of $c$, satisfying the following Leibniz rule:
\[
c(\alpha_1, \dots, \alpha_{k-1}, f \alpha_k) = \sigma_c (\alpha_1, \dots, \alpha_{k-1})(f) \alpha_k + f c(\alpha_1, \dots, \alpha_k),
\]
for all $\alpha_1, \dots, \alpha_k \in \Gamma(E)$, $f \in C^{\infty}(M)$.
\end{definition}

$1$-derivations are simply \emph{derivations}, $2$-derivations are called \emph{biderivations}. The space of derivations of $E$ is denoted by $\mathfrak D(E)$ (or $\mathfrak D (E, M)$ if we want to insist on the fact that the base of the vector bundle $E$ is $M$). The space of $k$-derivations is denoted $\mathfrak D^k (E)$ (or $\mathfrak D^k (E,M)$). In particular, $\mathfrak D^1 (E) = \mathfrak D (E)$. We also put $\mathfrak D^0 (E) = \Gamma (E)$ and $\mathfrak D^\bullet (E) = \bigoplus_{k \geq 0} \mathfrak D^k (E)$. Then $\mathfrak D^{\bullet}(E)[1]$, endowed with the \emph{Gerstenhaber bracket} $\llbracket -,- \rrbracket$, is a graded Lie algebra. We recall that, for $c_1 \in \mathfrak D^k(E)$, and $c_2 \in \mathfrak D^l (E)$, the \emph{Gerstenhaber product} of $c_1$ and $c_2$ is the $\mathbb R$-$(k+l-1)$-linear map $c_1 \circ c_2$ given by
\[
(c_1 \circ  c_2) (\alpha_1, \dots, \alpha_{k+l-1}) = \sum_{\tau \in S_{l,k-1}} (-1)^{\tau} c_1(c_2 (\alpha_{\tau(1)}, \dots, \alpha_{\tau(l)}), \alpha_{\tau(l+1)}, \dots, \alpha_{\tau(l+k-1)}),
\]
for all $\alpha_1, \dots, \alpha_{k+l-1} \in \Gamma(E)$, and the Gerstenhaber bracket is defined by
\[
\llbracket c_1, c_2 \rrbracket = (-1)^{(k-1)(l-1)} c_1 \circ c_2 - c_2 \circ c_1.
\]
The graded Lie algebra $\mathfrak D^\bullet (E)[1]$ first appeared in \cite{grab:Lie}.

The group of vector bundle automorphisms of $E$ acts naturally on multiderivations of $E$. If $\phi: E \to E$ is an automorphism covering the diffeomorphism $\phi_M : M \to M$, then $\phi$ acts on sections of $E$ (by pull-back) via the following formula:
\[
\phi^* \alpha := \phi^{-1} \circ \alpha \circ \phi_M, \quad \alpha \in \Gamma (E),
\]
and it acts on higher degree multiderivations via:
\[
(\phi^* c) (\alpha_1, \dots, \alpha_{k}) := \phi^* \left(c(\phi^{-1}{}^* \alpha_1, \dots, \phi^{-1}{}^* \alpha_{k})\right)
\]
for all $ \alpha_1, \dots, \alpha_{k} \in \Gamma(E), c \in \mathfrak D^k (E)$. Moreover, $\phi$ acts in the obvious way on sections of the dual bundle $E^*$. It is clear that
\begin{equation}\label{eq:phi^*}
\begin{aligned}
\phi^*(f \alpha) & = \phi_M^* f \cdot \phi^* \alpha, \\
\phi^* (c(\alpha_1, \dots, \alpha_k)) & = (\phi^* c) (\phi^* \alpha_1, \dots, \phi^* \alpha_k), \\
\phi_M^* \langle \varphi, \alpha \rangle & = \langle \phi^* \varphi, \phi^* \alpha \rangle,
\end{aligned}
\end{equation}
for all $\alpha, \alpha_1, \ldots, \alpha_k \in \Gamma (E)$, $f \in C^\infty (M)$, and $\varphi \in \Gamma (E^\ast)$, where $\langle -,- \rangle : E^\ast \otimes E \to \mathbb R$ is the duality pairing. 
Finally, $\phi$ acts on the exterior algebras of $E$ and $E^\ast$, and it also acts on vector bundle maps $\wedge^\bullet E \to TM$ in the obvious way.

A direct computation shows that the action of vector bundle automorphisms on multiderivations does also respect the Gerstenhaber bracket, i.e. 
\begin{equation}\label{eq:pb_Gerst}
\phi^* \llbracket c_1, c_2 \rrbracket = \llbracket \phi^* c_1, \phi^* c_2 \rrbracket
\end{equation}
for all $c_1, c_2 \in \mathfrak D^\bullet(E)$. Additionally,
\begin{equation}\label{eq:phi^*sigma}
\phi^* \sigma_c = \sigma_{\phi^* c}.
\end{equation}
for all $c \in \mathfrak D^\bullet(E)$.

If $A \Rightarrow M$ is a Lie algebroid, the Lie bracket $b_A = [-,-]$ on sections of $A$ is a biderivation and it contains the full information about $A \Rightarrow M$. Additionally, $\llbracket b_A, b_A \rrbracket = 0$ as a consequence of the Jacobi identity. We summarize this remark with the following

\begin{proposition}\label{prop:def_MC}
Lie algebroid structures on $A \to M$ are in one-to-one correspondence with Maurer-Cartan elements in the graded Lie algebra $\mathfrak D^\bullet (A)[1]$, i.e. ~degree $1$ elements $b$ such that $\llbracket b, b \rrbracket = 0$.
\end{proposition}

Now, fix a Lie algebroid structure $A \Rightarrow M$ on the vector bundle $A \to M$, and let $b_A$ be the Lie bracket on sections of $A$. Equipped with the Gerstenhaber bracket and the interior derivation $\delta := \llbracket b_A, - \rrbracket$, $\mathfrak D^\bullet (A)[1]$ is a differential graded Lie algebra (DGLA), denoted $C^\bullet_{\mathrm{def}} (A)$ (or $C^\bullet_{\mathrm{def}} (A, M)$ if we want to insist on the base manifold being $M$) and called the \emph{deformation complex of $A$}. The cohomology of $C^\bullet_{\mathrm{def}} (A)$ is denoted $H^\bullet_{\mathrm{def}} (A)$ (or $H^\bullet_{\mathrm{def}} (A, M)$), and called the \emph{deformation cohomology} of $A$.

\begin{remark} Notice that we adopt a different convention from \cite{crainic:def}, where $C^k_{\mathrm{def}}(A)$ is the space of $k$-derivations. With that convention, however, $C^{\bullet}_{\mathrm{def}}(A)$ is a DGLA only up to a shift. \end{remark}

The differential $\delta: C^{\bullet}_{\mathrm{def}}(A) \to C^{\bullet+1}_{\mathrm{def}}(A)$ is given, on $k$-derivations, by
\begin{equation}\label{eq:diff}
\begin{aligned}
\delta c (\alpha_1, \dots, \alpha_{k+1}) = & \sum_{i} (-1)^{i+1} [\alpha_i, c(\alpha_1, \dots, \widehat{\alpha_i}, \dots, \alpha_{k+1})] \\
& + \sum_{i<j} (-1)^{i+j} c([\alpha_i, \alpha_j], \alpha_1, \dots, \widehat{\alpha_i}, \dots, \widehat{\alpha_j}, \dots, \alpha_{k+1}).
\end{aligned}
\end{equation}

\begin{definition} A  \emph{deformation} of $b_A$ is a(n other) Lie algebroid structure on the vector bundle $A \to M$. \end{definition}

It is clear that $b = b_A + c$ satisfies $\llbracket b, b \rrbracket = 0$ if and only if 
\[
\delta c + \dfrac{1}{2} \llbracket c, c \rrbracket = 0,	
\]
i.e.~ {$c$ is a (degree $1$) solution of the  \emph{Maurer-Cartan equation} in the DGLA $C^{\bullet}_{\mathrm{def}}(A)$}. Hence Proposition \ref{prop:def_MC} can be rephrased saying that  \emph{deformations of $b_A$ are in one-to-one correspondence with Maurer-Cartan elements of $C^{\bullet}_{\mathrm{def}}(A)$}.

Now, let $b_0, b_1$ be deformations of $b_A$. We say that $b_0$ and $b_1$ are  \emph{equivalent} if there exists a  \emph{fiber-wise linear} isotopy taking $b_0$ to $b_1$, i.e. there is a smooth path of vector bundle automorphisms $\phi_t: A \to A$, $t \in [0,1]$, such that $\phi_0 = \operatorname{id}_A$ and $\phi_1^\ast b_1 = b_0$. On the other hand, two Maurer-Cartan elements $c_0, c_1$ are \emph{gauge-equivalent} if they are interpolated by a smooth path of $1$-cochains $c_t$, and $c_t$ is a solution of the following ODE: 
\begin{equation}\label{eq:c_t}
\frac{dc_t}{dt} = \delta \Delta_t + \llbracket c_t, \Delta_t \rrbracket,
\end{equation}
for some smooth path of $0$-cochains (i.e. derivations) $\Delta_t$, $t \in [0,1]$. 

Notice that (\ref{eq:c_t}) is equivalent to
\begin{equation}\label{eq:b_t}
\frac{db_t}{dt} = \llbracket b_t, \Delta_t \rrbracket
\end{equation}
where $b_t = b_A + c_t$.

\begin{proposition}\label{prop:eq_G_eq}
Let $b_0 = b_A + c_0, b_1 = b_A + c_1$ be deformations of $b_A$. If $b_0, b_1$ are equivalent, then $c_0, c_1$ are gauge-equivalent. If $M$ is compact, the converse is also true.
\end{proposition}

\begin{proof}

Suppose that $b_0$ and $b_1$ are equivalent deformations, and let $\phi_t : A \to A$ be a fiber-wise linear isotopy taking $b_0$ to $b_1$. Set $b_t = \phi_t^{-1}{}^* b_0 = b_A + c_t$, and let $\Delta_t$ be the infinitesimal generator of $\phi_t$, i.e.
\begin{equation}\label{eq:xi_t}
\dfrac{d \phi_t^*}{dt} = \phi_t^* \circ \Delta_t.
\end{equation}

Notice that 
\[
\llbracket b_t, b_t \rrbracket = \llbracket \phi_t^{-1}{}^* b_0, \phi_t^{-1}{}^* b_0 \rrbracket = \phi_t^{-1}{}^* \llbracket b_0, b_0 \rrbracket = 0,
\]
 so $b_t$ is a deformation of $b_A$ for all $t$. Moreover, we have $\phi_t^*(b_t(\alpha, \beta)) = b_0(\phi_t^* \alpha, \phi_t^* \beta)$ for all $\alpha, \beta \in \Gamma(A)$. Differentiating with respect to $t$, we obtain:
\[
\begin{aligned}
\phi_t^* \left( \Delta_t(b_t(\alpha,\beta)) + \dfrac{db_t}{dt}(\alpha,\beta) \right) & = b_0\left(\phi_t^* (\Delta_t (\alpha)), \phi_t^*\beta \right) + b_0\left(\phi_t^* \alpha, \phi_t^*(\Delta_t(\beta))\right) \\ 
& =  \phi_t^* \left(b_t(\Delta_t(\alpha),\beta) + b_t(\alpha, \Delta_t(\beta))\right),
\end{aligned}
\]
so
\[
\dfrac{db_t}{dt}(\alpha,\beta) = b_t(\Delta_t(\alpha), \beta) + b_t(\alpha, \Delta_t(\beta)) - \Delta_t(b_t(\alpha,\beta)),
\]
i.e.~(\ref{eq:b_t}), hence (\ref{eq:c_t}), holds, as desired.

Conversely, suppose that $M$ is compact and there exist a family of derivations $\Delta_t$ and a family of $1$-cochains $b_t$ such that (\ref{eq:c_t}) or, equivalently, (\ref{eq:b_t}) holds. Let $X_t$ be the symbol of $\Delta_t$. From compactness, $X_t$ is a complete time-dependent vector field on $M$, i.e.~it generates a complete flow $(\phi_M)_t$. The time dependent derivation $\Delta_t$ generates a flow by vector bundle automorphisms $\phi_t : A \to A$, covering the complete flow $(\phi_M)_t$ (and implicitly defined by the ODE (\ref{eq:xi_t})).  By linearity, $\phi_t$ is a complete flow itself. We want to show that
\begin{equation}\label{eq:Xi_t}
\phi_t^*(b_t(\alpha, \beta)) = b_0(\phi_t^* \alpha, \phi_t^* \beta) \quad \alpha,\beta \in \Gamma (A).
\end{equation}
For $t = 0$ this is obviously true and the derivatives of both sides are the same because of (\ref{eq:b_t}). So we have (\ref{eq:Xi_t}), and, by taking $t = 1$, we conclude that $\phi_t$ is a (fiber-wise linear) isotopy taking $b_0$ to $b_1$.
\end{proof}

\begin{remark}
An  \emph{infinitesimal deformation} of a Lie algebroid $A \Rightarrow M$ is an element $c \in C^1_{\mathrm{def}}(A)$ such that $\delta c = 0$, i.e.~a 1-cocycle in $C^{\bullet}_{\mathrm{def}}(A)$.  {As usual in deformation theory}, this definition is motivated by the fact that, \emph{if $c_t$ is a smooth path of Maurer-Cartan elements starting at $0$, then $\frac{dc_t}{dt}|_{t=0}$ is an infinitesimal deformation of $A$}.   More generally, the cocycle condition $\delta c = 0$ is just the linearization at $c = 0$ of the Maurer-Cartan equation. Hence, $1$-cocycles in $C^\bullet_{\mathrm{def}} (A)$ can be seen as the (formal) tangent vectors to the variety of Maurer-Cartan elements. Similarly, $1$-coboundaries can be seen as tangent vectors to the gauge orbit through $0$. We conclude that $H^1_{\mathrm{def}}(A)$ is the \emph{formal tangent space} to the moduli space of deformations under gauge equivalence. 
\end{remark}

\begin{remark} 
The deformation complex of a Lie algebroid has an efficient description in terms of graded geometry. In fact, graded geometry becomes very useful when dealing with several issues related to VB-algebroids.

Let $A \Rightarrow M$ be a Lie algebroid and let $(\Omega^{\bullet}_A = \Gamma(\wedge^{\bullet} A^*), d_A)$ be its de Rham complex (sometimes we will use $\Omega^{\bullet}_{A, M}$ for $A$-forms, if we want to insist on $M$ being the base manifold). Cochains in $\Omega^{\bullet}_A$ can be seen as functions on the DG-manifold $A[1]$ obtained from $A$ shifting by one the fiber degree. The $Q$-structure on $A[1]$ is simply $d_A$. Additionally, there is a canonical isomorphism $C^{\bullet}_{\mathrm{def}}(A) \cong \mathfrak{X}(A[1])^{\bullet}$ of DGLAs, where $\mathfrak X (A[1])^{\bullet}$ is the space of vector fields on the DG-manifold $A[1]$ (in other words, $\mathfrak X (A[1])^{\bullet}$ is the space of (graded) derivations of $\Omega^\bullet_A$). With the graded commutator and the adjoint operator $[d_A, -]$, $\mathfrak X (A[1])^\bullet$ is indeed a DGLA. Isomorphism $C^{\bullet}_{\mathrm{def}}(A) \to \mathfrak{X}(A[1])^{\bullet}$, $c \mapsto \delta_c$ can be described explicitly as follows. Let $c \in C^k (A)$ and let $\sigma_c$ be the symbol of $c$. Then $\delta_c \in \mathfrak{X}(A[1])^{\bullet}$  {is the degree $k$ vector field that takes} $\omega \in \Omega^p_A$, to $\delta_c \omega \in \Omega^{k + p}_A$ with
\begin{equation}\label{eq:def_der}
\begin{aligned}
\delta_c \omega (\alpha_1, \dots, \alpha_{k+p}) = & \sum_{\tau \in S_{k,p}} (-1)^{\tau} \sigma_c (\alpha_{\tau(1)}, \dots, \alpha_{\tau(k)}) \omega (\alpha_{\tau(k+1)}, \dots, \alpha_{\tau(k+p)}) \\
& - \sum_{\tau \in S_{k+1,p-1}} (-1)^{\tau} \omega (c(\alpha_{\tau(1)}, \dots, \alpha_{\tau(k+1)}), \alpha_{\tau(k+2)}, \dots, \alpha_{\tau(k+p)})
\end{aligned}
\end{equation}
where $S_{l,m}$ denotes $(l,m)$-unshuffles.
Notice that $c$ can be reconstructed from $\delta_c$ by using formula (\ref{eq:def_der}) for $p = 0, 1$:
\begin{equation}\label{eq:delta_c_f}
 \delta_c f (\alpha_1, \dots, \alpha_k)  = \sigma_c(\alpha_1, \dots, \alpha_{k})f,
 \end{equation}
 and
 \begin{equation}\label{eq:delta_c_phi}
\delta_c \varphi (\alpha_1, \dots, \alpha_{k+1})  = \sum_{i} (-1)^{k-i} \sigma_c (\alpha_1, \dots, \widehat{\alpha_i}, \dots, \alpha_{k+1})\langle \varphi, \alpha_i \rangle + \langle \varphi, c(\alpha_1, \dots, \alpha_{k+1}) \rangle,
\end{equation}
where $f \in C^{\infty}(M)$, $\varphi \in \Omega^1_{A} = \Gamma (A^\ast)$, and $\alpha_1, \dots, \alpha_{k+1} \in \Gamma(A)$.
\end{remark}
 

\subsection{Double vector bundles and VB-algebroids}\label{Sec:deformations2}

In this section we recall the basic definitions and properties of double vector bundles and VB-algebroids that will be useful later. For all the necessary details about the homogeneity structure of a vector bundle, including our notations, we refer to Appendix \ref{Sec:homogeneity}, and we recommend the reader to read the appendix before going on with the bulk of the paper. We only recall here that, given a vector bundle $E \to M$, the \emph{homogeneity structure} of $E$ is the action $h : \mathbb R_{\geq 0} \times E \to E$, $(\lambda , e) \mapsto h_\lambda e := \lambda \cdot e$, of non-negative reals on $E$ by homotheties (fiber-wise multiplication by scalars).

\begin{definition} A  \emph{double vector bundle} (DVB for short) is a vector bundle in the category of vector bundles.  More precisely, it is a commutative square
\begin{equation}\label{eq:DVB}
\begin{array}{c}
\xymatrix{
W \ar[d]_{q_W} \ar[r]^{\tilde p} & E \ar[d]^{q} \\
A \ar[r]^{p} & M}
\end{array},
\end{equation}
where all four sides are vector bundles, the projection $q_W: W \to A$, the addition $+_A: W \times_A W \to W$, the multiplication $\lambda\, \cdot_A : W \to W$ by any scalar $\lambda \in \mathbb R$ in the fibers of $W \to A$ and the zero section $\tilde{0}^A: A \to W$ are vector bundle maps covering the projection $q: E \to M$, the addition $+ : E \times_M E \to E$, the scalar multiplication $\lambda \cdot {}: E \to E$ and the zero section $0^E: M \to E$, respectively. The projection, the addition, the scalar multiplication and the zero section of a vector bundle will be called the \emph{structure maps}. DVB (\ref{eq:DVB}) will be also denoted by $(W \to E; A \to M)$. \end{definition}

Notice that $W$ is a vector bundle over $E$ and over $A$, so it carries two homogeneity structures. However, we will mainly use the latter and denote it simply by $h$. 
For more many details on DVBs we refer to \cite{mackenzie} and \cite{gracia-saz:vb}.

Let $(W \to E; A \to M)$ be a DVB. The manifold $W$ will be called the \emph{total space}.
Consider the submanifold
\[
C := \operatorname{ker}(W \to E) \cap \operatorname{ker}(W \to A) \subset W.
\]
In other words, elements of $C$ are those projecting simultaneously on the (images of the) zero sections of $A$ and $E$ (which are both diffeomorphic to $M$). The fiber-wise operations of the vector bundles $W \to E$ and $W \to A$ coincide on $C$ (see \cite{mackenzie}), so they define a (unique) vector bundle structure on $C$ over $M$. The vector bundle $C \to M$ is called the \emph{core} of $(W \to E; A \to M)$.

In the following, we denote by $\Gamma (W, E)$ the space of sections of $W \to E$. Sections of $C \to M$ can be naturally embedded into $\Gamma(W,E)$, via the map $\Gamma(C) \to \Gamma(W,E)$, $\chi \mapsto \overline \chi$, defined by: 

\begin{equation}\label{eq:core_incl}
\overline \chi_e = \tilde{0}{}^E_e +_A \chi_{q(e)}, \quad e \in E.
\end{equation}
The image of the inclusion $\chi \mapsto \overline \chi$ is, by definition, the space $\Gamma_{\mathrm{core}}(W,E)$ of \emph{core sections} of $W \to E$. 

There is another relevant class of sections of $W \to E$: \emph{linear sections}. We say that a section of $W \to E$ is a  \emph{linear section} if it is a vector bundle map covering some section of $A \to M$. The space of linear sections of $W \to E$ is denoted $\Gamma_{\mathrm{lin}}(W,E)$.  
We will usually denote by $\tilde \alpha, \tilde \beta, \ldots$ the sections in $\Gamma_{\mathrm{lin}}(W,E)$. The $C^\infty (E)$-module $\Gamma (W, E)$ is spanned by $\Gamma_{\mathrm{core}}(W,E)$ and $\Gamma_{\mathrm{lin}}(W,E)$.

Linear and core sections of $W \to E$ can be efficiently characterized using the homogeneity structure $h$. Namely, the following lemma holds.

\begin{lemma}\label{prop:char_lin}
A section $w \in \Gamma(W,E)$ is
\begin{enumerate}
\item linear if and only if $h_\lambda^* w = w$ for every $\lambda > 0$;
\item core if and only if $h_\lambda^* w = \lambda^{-1} w$ for every $\lambda > 0$.
\end{enumerate}
\end{lemma}

More generally, we say that a section $w$ of $W \to E$ is \emph{of weight $q$} if $h_\lambda^* w = \lambda^q w$ for every $\lambda > 0$. Using this terminology, linear sections are precisely sections of weight $0$ and core sections are sections of weight $-1$. It is easy to check that \emph{there are no non-zero sections of $W \to E$ of weight less than $-1$.}


\begin{remark}
Let $(W \to E;A \to M)$ be a DVB, let $C$ be its core and let $W^*_A \to A$ be the dual vector bundle of $W \to A$. Then
\[
\xymatrix{
W^*_A \ar[d] \ar[r] & C^* \ar[d] \\
A \ar[r] & M}
\]
is a DVB, called the  \emph{dual} of $W$ over $A$, whose core is $E^*$. We refer to \cite{mackenzie} for the structure maps of the dual DVB.
\end{remark}

\begin{examplex} A distinguished example of a DVB is the  \emph{tangent double of a vector bundle}. If $E \to M$ is a vector bundle, then
\[
\xymatrix{
TE \ar[d] \ar[r] & E \ar[d] \\
TM \ar[r] & M}
\]
is a DVB with core canonically isomorphic to $E$. It is easy to see that linear sections of $TE \to E$ are precisely linear vector fields (see the appendix). Moreover, the inclusion $\Gamma (E) \to \Gamma_{\mathrm{core}}(TE \to E)$ is the classical \emph{vertical lift}, identifying a section of $E$ with a fiber-wise constant vertical vector field on $E$ itself. We will also call \emph{core vector fields} the core sections of $TE \to E$.

The dual of $TE$ over $E$ is
\[
\begin{array}{c}
\xymatrix{
T^*E \ar[d] \ar[r] & E \ar[d] \\
E^\ast \ar[r] & M}
\end{array}.
\]
\end{examplex}

We now pass to VB-algebroids.

\begin{definition}\label{def:vb-alg}
A \emph{VB-algebroid} is a DVB as in (\ref{eq:DVB}), equipped with a Lie algebroid structure $W \Rightarrow E$ such that the anchor $\rho_W: W \to TE$ is a vector bundle map covering a vector bundle map $\rho_A: A \to TM$ and the Lie bracket $[-,-]_W$ on sections of $W \to E$ satisfies
\begin{equation}\label{eq:VB-alg}
\begin{aligned}
{}[\Gamma_{\mathrm{lin}}(W,E), \Gamma_{\mathrm{lin}}(W,E)]_W & \subset \Gamma_{\mathrm{lin}}(W,E), \\
[\Gamma_{\mathrm{lin}}(W,E), \Gamma_{\mathrm{core}}(W,E)]_W & \subset \Gamma_{\mathrm{core}}(W,E), \\
[\Gamma_{\mathrm{core}}(W,E), \Gamma_{\mathrm{core}}(W,E)]_W & = 0.
\end{aligned}
\end{equation}
\end{definition}

Notice that, using the grading defined above, Property (\ref{eq:VB-alg}) is equivalent to asking that the Lie bracket on $\Gamma (W, E)$ is of weight 0. This can be made very precise using the action of vector bundle automorphisms on multiderivations (see below).

\begin{remark}
Let $(W \Rightarrow E; A \Rightarrow M)$ be a VB-algebroid with core $C$, and let $(W^\ast \to C^\ast; A \to M)$ be its dual DVB. One can show that there is a canonical  VB-algebroid structure $(W^*_A \Rightarrow C^\ast; A \Rightarrow M)$ on the latter, called the \emph{dual VB-algebroid}. The dual VB-algebroid will appear only marginally in the sequel, so we do not discuss the details of this construction. For more information, see \cite{mackenzie:ehresmann}.
\end{remark}



\subsubsection{Graded geometric description}\label{subsec:alt_desc}
 
There is a very useful description of VB-algebroids in terms of graded geometry. We begin discussing \emph{linear vector fields} on (the total space of) a vector bundle $\mathcal E \to \mathcal M$ of graded manifolds. First we fix our notation. As already mentioned, a section $\phi$ of the dual bundle $\mathcal E^\ast \to \mathcal M$ determines a fiber-wise linear function  $\ell_\phi$ on $\mathcal E$. As in the non-graded case, a section $\varepsilon$ of $\mathcal E$ itself determines a \emph{fiber-wise constant} vector field $\varepsilon^\uparrow \in \mathfrak X (\mathcal E)^{\bullet}$, its \emph{vertical lift}, uniquely defined by
\[
\varepsilon^\uparrow (\ell_\phi) := (-)^{|\varepsilon||\phi|} \langle \phi, \varepsilon \rangle.
\] 
We denote by $\mathfrak X_{\mathrm{core}} (\mathcal E)^{\bullet}$ the space of core vector fields, i.e.~fiber-wise constant vertical vector fields on $\mathcal E$. Correspondence $\varepsilon \mapsto \varepsilon^\uparrow$ establishes a graded $C^\infty (\mathcal M)^\bullet$-module isomorphism $ \Gamma (\mathcal E)^\bullet \cong \mathfrak X_{\mathrm{core}} (\mathcal E)^{\bullet}$. Now let $X \in \mathfrak X (\mathcal E)^{\bullet}$. Then $X$ is \emph{linear} if it preserves fiber-wise linear functions. Equivalently, $X$ is linear if the (graded) commutator $[X, -]$ preserves fiber-wise constant vector fields. We denote by $\mathfrak X_{\mathrm{lin}}(\mathcal E)^{\bullet}$ the space of linear vector fields on $\mathcal E$. Notice that linear vector fields do also preserve fiber-wise constant functions. Finally, similarly as in the nongraded case, denote by $\mathfrak D(\mathcal E)^{\bullet} $ the space of  graded derivations of $\mathcal E$. There is a canonical isomorphism of graded Lie algebras and graded $C^\infty (\mathcal M)^\bullet$-modules  $\mathfrak X_{\mathrm{lin}}(\mathcal E)^{\bullet} \to \mathfrak D(\mathcal E)^{\bullet} $, $X \mapsto D_X$, implicitly defined by $(D_X \varepsilon)^\uparrow = [X, \varepsilon^\uparrow]$, for all $\varepsilon \in \Gamma (\mathcal E)^\bullet$.

Now, we have already recalled that Lie algebroids are equivalent to DG-manifolds concentrated in degree $0$ and $1$. For VB-algebroids we have an analogous result \cite{voronov:q} that we now briefly explain. Recall that a DG-vector bundle is a vector bundle of graded manifolds $\mathcal E \to \mathcal M$ such that $\mathcal E$ and $\mathcal M$ are both DG-manifolds, with homological vector fields $Q_{\mathcal E}$ and $Q_{\mathcal M}$, respectively, and, additionally, $Q_{\mathcal E}$ is linear, and projects onto $Q_{\mathcal M}$. Equivalently $\mathcal E \to \mathcal M$ is a vector bundle of graded manifolds, $\mathcal M$ is a DG-manifold, with homological vector field $Q_{\mathcal M}$,  $\mathcal E$ is equipped with a homological derivation $D_{\mathcal E}$, i.e.~a degree $1$ derivation such that $[D_{\mathcal E}, D_{\mathcal E}] = 0$, and, additionally, the symbol of $D_{\mathcal E}$ is precisely $Q_{\mathcal M}$. For more details about DG-vector bundles see, e.g.~\cite{Vit:vv-forms}.

Finally, let $(W \to E; A \to M)$ be a DVB. If we shift the degree in the fibers of both $W \to E$ and $A \to M$ (and use the functoriality of the shift) we get a vector bundle of graded manifolds, denoted $W[1]_E \to A[1]$. If $(W \Rightarrow E; A \Rightarrow M)$ is a VB-algebroid, then $W[1]_E \to A[1]$ is a DG-vector bundle concentrated in degree $0$ and $1$.
 
\begin{theorem} [see \cite{voronov:q}] \label{prop:grad_VB}
Correspondence $(W \Rightarrow E; A \Rightarrow M)  \rightsquigarrow (W[1]_E \to A[1])$ establishes an equivalence between the category of VB-algebroids and the category of DG-vector bundles concentrated in degree 0 and 1.
\end{theorem}

\subsection{The linear deformation complex of a VB-algebroid}\label{sec:lin_def}

In this subsection we introduce the main object of this paper: the \emph{linear deformation complex of a VB-algebroid}, first introduced in \cite{etv:infinitesimal} (for different purposes from the present ones). Actually, the whole discussion in Section \ref{Sec:deformations1} extends to VB-algebroids. We skip most of the proofs: they can be carried out in a very similar way as for plain Lie algebroids.

We begin with a DVB $(W \to E;A \to M)$. Denote by $\mathfrak D^\bullet (W, E)$ the space of multiderivations of the vector bundle $W \to E$. As in Subsection \ref{Sec:deformations2}, denote by $h$ the homogeneity structure of $W \to A$. The action of $h$ induces a grading on the space of multiderivations.

\begin{definition} A multiderivation $c \in \mathfrak D^\bullet (W, E)$ is \emph{homogeneous of weight $q$} (or, simply, of \emph{weight} $q$) if  $h_\lambda^* c = \lambda^q c$ for every $\lambda > 0$. A multiderivation is \emph{linear} if it is of weight $0$, and it is \emph{core} if it is of weight $-1$.

We denote by $\mathfrak D^\bullet_q (W,E)$ the space of multiderivations of weight $q$, and by $\mathfrak D^\bullet_{\mathrm{lin}}(W,E)$ and $\mathfrak D^\bullet_{\mathrm{core}}(W,E)$, respectively, the spaces of linear and core multiderivations.
\end{definition}

As $\Gamma_\mathrm{core}(W,E)$ and $\Gamma_\mathrm{lin}(W,E)$ generate $\Gamma(W,E)$, a multiderivation is completely characterized by its action, and the action of its symbol, on linear and core sections. From equation (\ref{eq:phi^*}) and the fact that there are no non-zero sections of weight less than $-1$, it then follows \emph{there are no non-zero multiderivations of weight less than $-1$.} Moreover:

\begin{proposition}\label{prop:lin_der}
Let $c$ be a $k$-derivation of $W \to E$. Then $c$ is linear if and only if all the following conditions are satisfied
\begin{enumerate}
\item $c(\tilde \alpha_1, \ldots, \tilde \alpha_k)$  is a linear section,
\item $c (\tilde \alpha_1, \ldots, \tilde \alpha_{k-1}, \overline \chi_1)$ is a core section,
\item $c (\tilde \alpha_1, \ldots, \tilde \alpha_{k-i}, \overline \chi_1, \ldots, \overline \chi_i) = 0$,
\item $\sigma_{c} (\tilde \alpha_1, \ldots, \tilde \alpha_{k-1})$ is a linear vector field,
\item $\sigma_{c} (\tilde \alpha_1, \ldots, \tilde \alpha_{k-2}, \overline \chi_1)$ is a core vector field,
\item $\sigma_{c}(\tilde \alpha_1, \ldots, \tilde \alpha_{k-i-1}, \overline \chi_1, \ldots, \overline \chi_i) = 0$
\end{enumerate}
for all linear sections $\tilde \alpha_1, \ldots, \tilde \alpha_k$, all core sections $\overline \chi_1, \ldots, \overline \chi_i$ of $W \to E$, and all $i \geq 2$.
\end{proposition}

\begin{proof} This follows immediately from Lemma \ref{prop:char_lin} (see also \cite{etv:infinitesimal}). \end{proof}

In particular, a linear $k$-derivation is uniquely determined by its action on $k$ linear sections and on $k-1$ linear sections and a core section, and by the action of its symbol on $k-1$ linear sections and on $k-2$ linear sections and a core section (see also \cite[Theorem 3.34]{etv:infinitesimal}). 

%

It immediately follows from (\ref{eq:pb_Gerst}) that $\mathfrak D^\bullet_{\mathrm{lin}} (W, E)[1]$ is a graded Lie subalgebra of $\mathfrak D^\bullet (W, E)[1]$. The following proposition is then straightforward.

\begin{proposition}\label{prop:def_VB}
VB-algebroid structures on the DVB $(W \to E; A \to M)$ are in one-to-one correspondence with Maurer-Cartan elements in $\mathfrak D^\bullet_{\mathrm{lin}} (W, E)[1]$.
\end{proposition}

Now, fix a VB-algebroid structure $(W \Rightarrow E; A \Rightarrow M)$ on the DVB $(W \to E; A \to M)$, and denote by $b_W$ the Lie bracket on sections of $W \to E$. We also denote by $C^\bullet_{\mathrm{def}} (W, E)$ the deformation complex of the top algebroid $W \Rightarrow E$. It is clear that $b_W$ is a linear biderivation of $W \to E$, i.e.~$b_W \in \mathfrak D^2_{\mathrm{lin}} (W, E)$. Hence $\mathfrak D^\bullet_{\mathrm{lin}} (W, E)[1]$ is a subDGLA of $C^\bullet_{\mathrm{def}} (W, E)$, denoted $C^{\bullet}_{\mathrm{def}, \mathrm{lin}}(W)$, and called the  \emph{linear deformation complex} of $W \Rightarrow E$. Its cohomology is denoted $H^\bullet_{\mathrm{def}, \mathrm{lin}} (W)$ and called the  \emph{linear deformation cohomology} of $(W \Rightarrow E; A \Rightarrow M)$.

\begin{definition} A  \emph{linear deformation} of $b_W$ (or simply a \emph{deformation}, if this does not lead to confusion) is a(n other) VB-algebroid structure on the DVB $(W \to E;A \to M)$. \end{definition}

Exactly as for Lie algebroids,  Proposition \ref{prop:def_VB} is equivalent to say that  \emph{deformations of $b_W$ are in one-to-one correspondence with Maurer-Cartan elements of $C^{\bullet}_{\mathrm{def}, \mathrm{lin}}(W)$.}

Let $b_0, b_1$ be linear deformations of $b_W$. We say that $b_0$ and $b_1$ are  \emph{equivalent} if there exists a DVB isotopy taking $b_0$ to $b_1$, i.e.~a smooth path of DVB automorphisms $\phi_t: W \to W$, $t \in [0,1]$ such that $\phi_0 = \operatorname{id}_W$ and $\phi_1^* b_1 = b_0$. On the other hand, two Maurer-Cartan elements $c_0, c_1$ in $C^\bullet_{\mathrm{def}, \mathrm{lin}}(W)$ are  \emph{gauge-equivalent} if they are interpolated by a smooth path of $1$-cochains $c_t \in C^\bullet_{\mathrm{def}, \mathrm{lin}}(W)$, and $c_t$ is a solution of the following ODE
\[
\dfrac{dc_t}{dt} = \delta \Delta_t + \llbracket c_t, \Delta_t \rrbracket,
\]  
for some smooth path of $0$-cochains $\Delta_t \in C^\bullet_{\mathrm{def}, \mathrm{lin}}(W)$, $t \in [0,1]$. Equivalently,
\[
\frac{db_t}{dt} = \llbracket b_t, \Delta_t \rrbracket,
\]
where $b_t = b_W + c_t$.

\begin{proposition}
The DGLA $C^{\bullet}_{\mathrm{def}, \mathrm{lin}}(W)$ controls deformations of the VB-algebroid $(W \Rightarrow E; A \Rightarrow M)$ in the following sense. Let $b_0 = b_W + c_0, b_1 = b_W + c_1$ be linear deformations of $b_W$. If $b_0, b_1$ are equivalent, then $c_0, c_1$ are gauge-equivalent. If $M$ is compact, the converse is also true. \end{proposition}

\begin{proof}
The proof is similar to that of Proposition \ref{prop:eq_G_eq}, with linear derivations replacing derivations and DVB automorphisms replacing vector bundle automorphisms. We only need to be careful when using the compactness hypothesis. Recall from \cite{etv:infinitesimal} that a linear derivation generates a flow by DVB automorphisms. In particular, if $\Delta_t$ is a time-dependent linear derivation of $W \to E$, then its symbol $X_t = \sigma (\Delta_t) \in \mathfrak X (E)$ is a \emph{linear vector field}, hence it generates a flow by vector bundle automorphisms of $E$. From the compactness of $M$, it follows that $X_t$, hence the flow of $\Delta_t$, is complete.
\end{proof} 
 
\begin{remark}
An  \emph{infinitesimal deformation} of $(W \Rightarrow E; A \Rightarrow M)$ is an element $c \in C^1_{\mathrm{def}, \mathrm{lin}}(W)$ such that $\delta c = 0$, i.e.~$c$ is a 1-cocycle in $C^{\bullet}_{\mathrm{def}, \mathrm{lin}}(W)$. If $c_t$ is a smooth path of Maurer-Cartan elements starting at $0$, then $\frac{dc_t}{dt}|_{t=0}$ is an infinitesimal deformation of $(W \Rightarrow E; A \Rightarrow M)$. Similarly as for Lie algebroids, $H^1_{\mathrm{def}, \mathrm{lin}}(W)$ is the formal tangent space to the moduli space of linear deformations under gauge equivalence. It also follows from standard deformation theory arguments that $H^2_{\mathrm{def}, \mathrm{lin}}(W)$ contains obstructions to the extension of an infinitesimal linear deformation to a formal one. Finally, we interpret $0$-degree deformation cohomologies. It easily follows from the definition that $0$-cocycles in $C^\bullet_{\mathrm{def}, \mathrm{lin}} (A)$ are \emph{infinitesimal multiplicative} (IM) derivations of $(W \Rightarrow E; A \Rightarrow M)$ i.e.~derivations of $W \to E$ generating a flow by VB-algebroid automorphisms \cite{etv:infinitesimal}. Among those, $1$-cocycles are \emph{inner IM derivations}, i.e.~IM derivations of the form $[\tilde \alpha, -]$ for some \emph{linear} section $\tilde \alpha$ of $W \to E$. So $H^0_{\mathrm{def}, \mathrm{lin}}(W)$ consists of \emph{outer IM derivations}. See \cite{etv:infinitesimal} for more details.
\end{remark}

\subsubsection{Alternative descriptions}

Let $(W\Rightarrow E; A \Rightarrow M)$ be a VB-algebroid. Then $W \Rightarrow E$ is a Lie algebroid. As in Subsection \ref{subsec:alt_desc} we denote by $W[1]_E$ the DG-manifold obtained from $W$ shifting by the degree in the fibers of $W \to E$. So $C^{\bullet}_{\mathrm{def}}(W) \cong \mathfrak{X}(W[1]_E)^{\bullet}$. Moreover, it is easy to see from Formula \ref{eq:def_der} and Proposition \ref{prop:lin_der} that a deformation cochain $c \in C^{\bullet}_{\mathrm{def}}(W)$ is linear if and only if the corresponding vector field $\delta_c \in \mathfrak{X}(W[1]_E)^{\bullet}$ is a linear vector field with respect to the vector bundle structure $W[1]_E \to A[1]$. So there is a canonical isomorphism of DGLAs
\[
C^{\bullet}_{\mathrm{def}, \mathrm{lin}}(W) \cong \mathfrak{X}_{\mathrm{lin}}(W[1]_E)^{\bullet}.
\]
As linear vector fields are equivalent to derivations, we also get
\begin{equation}\label{eq:der_grad}
C^{\bullet}_{\mathrm{def}, \mathrm{lin}}(W)  \cong \mathfrak D(W[1]_E, A[1])^{\bullet}
\end{equation}
as DGLAs.

\subsubsection{Deformations of $A$ from linear deformations of $W$}

There is a natural surjection $C^{\bullet}_{\mathrm{def}, \mathrm{lin}}(W) \to C^{\bullet}_{\mathrm{def}}(A)$ which is easily described in the graded geometric picture: it is just the projection 
\[
\mathfrak{X}_{\mathrm{lin}}(W[1]_E)^{\bullet} \to \mathfrak{X}(A[1])^{\bullet},
\]
of linear vector fields on the base. Equivalently, it is the symbol map
\[
\sigma : \mathfrak{D}(W[1]_E, A[1])^{\bullet} \to \mathfrak{X}(A[1])^{\bullet}.
\]
In particular, we get a short exact sequence of DGLAs
\begin{equation}\label{eq:SESDGLAs}
0 \to \operatorname{\operatorname{\mathfrak{End}}}( W[1]_E)^\bullet \to \mathfrak{X}_{\mathrm{lin}}(W[1]_E)^{\bullet} \to \mathfrak{X}(A[1])^{\bullet} \to 0,
\end{equation}
where $\operatorname{\mathfrak{End}}( W[1]_E)^\bullet$ is the space of (graded) endomorphisms of $W[1]_E \to A[1]$. Equivalently, there is a short exact sequence
\begin{equation}\label{eq:SESDGLAs_II}
0 \to \operatorname{\operatorname{\mathfrak{End}}}( W[1]_E)^\bullet \to C^\bullet_{\mathrm{def},\mathrm{lin}}(W) \to C^\bullet_{\mathrm{def}}(A) \to 0.
\end{equation}
Notice that the subDGLA $\operatorname{\operatorname{\mathfrak{End}}}( W[1]_E)^\bullet$ \emph{controls deformations of $(W \Rightarrow E; A \Rightarrow M)$ that fix $A \Rightarrow M$}, i.e.~deformations of $W$ that fix $b_A$ (the Lie algebroid structure on $A$) identify with Maurer-Cartan elements in $\operatorname{\operatorname{\mathfrak{End}}}( W[1]_E)^\bullet$. Finally, we obtain a long exact sequence
\begin{equation}\label{eq:LES_VB}
\dots \to H^k(\operatorname{\mathfrak{End}}(W[1]_E)) \to H^k_{\mathrm{def}, \mathrm{lin}}(W) \to H^k_{\mathrm{def}}(A) \to H^{k+1}(\operatorname{\mathfrak{End}}(W[1]_E)) \to \cdots
\end{equation}
connecting the linear deformation cohomology of $W$ with the deformation cohomology of $A$.

\begin{remark}
We will not need a description of the subcomplex $\operatorname{\mathfrak{End}}( W[1]_E )^\bullet\subset \mathfrak{X}_{\mathrm{lin}}(W[1]_E)^{\bullet}$ in terms of more classical data in this paper. However, we stress that this description exists in analogy with \cite[Theorem 3.34]{etv:infinitesimal}.
\end{remark}

\subsubsection{Deformations of the dual VB-algebroid}

We conclude this section noticing that the linear deformation complex of a VB-algebroid is canonically isomorphic to that of its dual. Let $(W \Rightarrow E;A \Rightarrow M)$ be a VB-algebroid with core $C$, and let $(W^\ast_A \Rightarrow C^\ast ; A \Rightarrow M)$ be the dual VB-algebroid.

\begin{theorem}
There is a canonical isomorphism of DGLAs $C^{\bullet}_{\mathrm{def}, \mathrm{lin}}(W) \cong C^{\bullet}_{\mathrm{def}, \mathrm{lin}}(W^*_A)$.
\end{theorem}

\begin{proof}
There is an easy proof exploiting graded geometry. We only sketch it, and leave the straightforward details to the reader. So, first of all, it is easy to see, e.g.~in local coordinates, that the vector bundles of graded manifolds $W^\ast_A [1]_{C^\ast} \to A[1]$ and $W[1]_E^\ast \to A[1]$ are actually isomorphic up to a shift in the degree of the fiber coordinates. Additionally, derivations of a vector bundle of graded manifolds are canonically isomorphic to that of
\begin{enumerate}
\item its dual,
\item any vector bundle obtained from it by a shift in the degree of the fibers.
\end{enumerate}
We conclude that
\[
C^{\bullet}_{\mathrm{def}, \mathrm{lin}}(W) \cong \mathfrak D(W[1]_E, A[1])^\bullet \cong \mathfrak D(W[1]_E^\ast, A[1])^\bullet \cong \mathfrak D(W^\ast[1]_{C^\ast}, A[1])^\bullet \cong C^\bullet_{\mathrm{def}, \mathrm{lin}}(W^\ast_A).
\]
\end{proof}

\subsection{From deformation cohomology to linear deformation cohomology}\label{sec:linearization_1}

Let $(W \Rightarrow E; A \Rightarrow M)$ be a VB-algebroid. We have shown that deformations of the VB-algebroid structure are controlled by a subDGLA $C^\bullet_{\mathrm{def}, \mathrm{lin}} (W)$ of the deformation complex $C^\bullet_{\mathrm{def}}(W)$ of the top Lie algebroid $W \Rightarrow E$. In the next section, we show that the inclusion $C^\bullet_{\mathrm{def}, \mathrm{lin}} (W) \hookrightarrow C^\bullet_{\mathrm{def}}(W)$ induces an inclusion $H^\bullet_{\mathrm{def}, \mathrm{lin}} (W) \hookrightarrow H^\bullet_{\mathrm{def}}(W)$ in cohomology. In particular, given an infinitesimal linear deformation that is trivial as infinitesimal deformation of the Lie algebroid $W \Rightarrow A$, i.e.~it is connected to the zero deformation by an infinitesimal isotopy of vector bundle maps, then it is also trivial as infinitesimal linear deformation, i.e.~it is also connected to the zero deformation by an infinitesimal isotopy of DVB maps.

The key idea is adapting to the present setting the ``homogenization trick'' of \cite{cabrera:hom}. Let $E \to M$ be a vector bundle. In their paper, Cabrera and Drummond consider the following natural projections from the $C^\infty (E)$ to its $C^\infty (M)$-submodules $C^\infty_q(E)$ (of \emph{weight $q$ homogeneous functions}):

\begin{equation}\label{eq:proj_fun}
\mathrm{pr}_q: C^\infty (E) \to C^\infty_q (E), \quad f \mapsto \dfrac{1}{q!} \dfrac{d^q}{d \lambda^q} |_{\lambda = 0} h_\lambda^* f.
\end{equation}

Notice that $\mathrm{pr}_q(f)$ is just the degree $q$ part of the (fiber-wise) Taylor polynomial of $f$. In the following, we adopt the notations from Appendix \ref{Sec:homogeneity} and denote
\begin{equation}\label{eq:core,lin}
\begin{aligned}
\mathrm{core} := & \ \mathrm{pr}_0: C^\infty (E) \to C^\infty_{\mathrm{core}} (E), \quad f \mapsto f_{\mathrm{core}} = h_0^* f, \\
\mathrm{lin} := & \ \mathrm{pr}_1: C^\infty (E) \to C^\infty_\mathrm{lin} (E), \quad f \mapsto f_{\mathrm{lin}} =  \dfrac{d}{d \lambda} |_{\lambda = 0} h_\lambda^* f,
\end{aligned}
\end{equation}
where $C^\infty_{\mathrm{core}} (E) := C^\infty_0 (E)$, and $C^\infty_\mathrm{lin} (E) := C^\infty_1 (M)$.

Formula (\ref{eq:proj_fun}) does not apply directly to multiderivations. To see why, let $(W \to E, A \to M)$ be a DVB, let $h$ be the homogeneity structure of $W \to A$, and let $c \in \mathfrak D^\bullet (W,E)$. Then the curve $\lambda \mapsto h_\lambda^* c$ is not defined in $0$. Actually, $\lambda = 0$ is a ``pole of order $1$'' for $h_\lambda^\ast c$. More precisely, we have the following:

\begin{proposition} 
The limit
\[
\lim_{\lambda \to 0} \lambda \cdot h_\lambda^* c
\]
exists and defines a core multiderivation $c_{\mathrm{core}}$.
\end{proposition}

\begin{proof} The existence of the limit can be shown in coordinates. Moreover, for every $\mu \neq 0$,
\[
h_\mu^* c_{\mathrm{core}} = h_\mu^* \left( \lim_{\lambda \to 0} \lambda \cdot h_\lambda^* c \right) = \lim_{\lambda \to 0} \lambda \cdot h^*_\mu h^*_\lambda c = \lim_{\lambda \to 0} \mu^{-1} \left( \lambda \mu \cdot h^*_{\lambda \mu} c \right) = \mu^{-1} c_{\mathrm{core}}. 
\]
\end{proof}

Next proposition can be proved in the same way.

\begin{proposition}\label{prop:c_0}
The limit
\[
\lim_{\lambda \to 0} \big(h_\lambda^* c - \lambda^{-1} \cdot c_{\mathrm{core}}\big)
\]
exists and defines a linear multiderivation $c_{\mathrm{lin}}$. \end{proposition}

So far we have defined maps
\begin{equation}\label{eq:projs}
\begin{aligned}
\mathrm{core} & {} :  \mathfrak D^\bullet (W,E) \to \mathfrak D_{\mathrm{core}}^\bullet (W,E), \quad  c \mapsto \lim_{\lambda \to 0} \lambda \cdot  h_\lambda^* c \\
\mathrm{lin} & {} : \mathfrak D^\bullet (W,E) \to \mathfrak D_\mathrm{lin}^\bullet (W,E), \quad  c \mapsto \lim_{\lambda \to 0} \big(h_\lambda^* c - \lambda^{-1} \cdot c_{\mathrm{core}}\big)
\end{aligned}
\end{equation}
that split the inclusions in $\mathfrak D^\bullet (W,E)$. We call the latter the \emph{linearization map}.

\begin{remark} Once we have removed the singularity at 0, we can proceed as in (\ref{eq:proj_fun}) and define the projections on homogeneous multiderivation of positive weights $q > 0$:
\[
\mathrm{pr}_q: \mathfrak D^\bullet (W,E) \to \mathfrak D^\bullet_q (W,E), \quad c \mapsto \dfrac{1}{q!} \dfrac{d^q}{d \lambda^q} |_{\lambda = 0} (h_\lambda^* c - \lambda^{-1} \cdot c_{\mathrm{core}}).
\]
\end{remark}

Now, let $(W \Rightarrow E, A \Rightarrow M)$ be a VB-algebroid. Then we have a linearization map
\[
\mathrm{lin}: C^\bullet_{\mathrm{def}}(W) \to C^\bullet_{\mathrm{def,lin}}(W).
\]

\begin{theorem}[Linearization of deformation cochains]\label{theor:lin}
The linearization map is a cochain map splitting the inclusion $C_{\mathrm{def}, \mathrm{lin}}^{\bullet}(W) \hookrightarrow C^{\bullet}_{\mathrm{def}}(W)$. In particular there is a direct sum decomposition
\[
C^{\bullet}_{\mathrm{def}}(W) \cong C^{\bullet}_{\mathrm{def}, \mathrm{lin}}(W) \oplus \ker(\mathrm{lin})^\bullet.
\]
of cochain complexes. Hence, the inclusion of linear deformation cochains into deformation cochains induces an injection
\begin{equation}\label{eq:incl}
H^{\bullet}_{\mathrm{def}, \mathrm{lin}}(W) \hookrightarrow H^{\bullet}_{\mathrm{def}}(W).
\end{equation}
\end{theorem}

\begin{proof} We only have to prove that the linearization preserves the differential $\delta = \llbracket b_W, - \rrbracket$ (here, as usual $b_W$ is the Lie bracket on sections of $W \Rightarrow E$). Using the fact that $b_W$ is linear, we have that $\delta$ commutes with $h_\lambda^*$. From (\ref{eq:diff}) it is obvious that $\delta$ preserves limits. So
\[
(\delta c)_{\mathrm{core}}  = \lim_{\lambda \to 0} \lambda \cdot h_\lambda^* (\delta c) = \lim_{\lambda \to 0} \lambda \cdot \delta(h_\lambda^* c) = \delta \left(\lim_{\lambda \to 0} \lambda \cdot h_\lambda^* c \right) = \delta c_{\mathrm{core}},
\]
and
\[
(\delta c)_\mathrm{lin}  = \lim_{\lambda \to 0} \left(h_\lambda^* (\delta c) - \lambda^{-1} \delta(c_{\mathrm{core}})\right) = \delta \left( \lim_{\lambda \to 0} (h_\lambda^* c - \lambda^{-1} \cdot c_{\mathrm{core}}) \right) = \delta c_{\mathrm{lin}},
\]
as desired. \end{proof}

The inclusion (\ref{eq:incl}) can be used to transfer vanishing results from deformation cohomology of the Lie algebroid $W \Rightarrow E$ to the linear deformation cohomology of the VB-algebroid $(W \Rightarrow E; A \Rightarrow M)$. For example, if $H^0_{\mathrm{def}}(W) = 0$, every Lie algebroid derivation of $W \Rightarrow E$ is inner, and hence every IM derivation of the VB-algebroid $W$ is inner. Similarly, if $W \Rightarrow E$ has no non-trivial infinitesimal deformations, so does $(W \Rightarrow E; A \Rightarrow M)$, and so on.

As a first example, consider a vector bundle $E \to M$. Then
\[
\xymatrix{
TE \ar[d] \ar@{=>}[r] & E \ar[d] \\
TM \ar@{=>}[r] & M}
\]
is a VB-algebroid.

\begin{proposition}\label{prop:TE}
The linear deformation cohomology of $(TE \Rightarrow E; TM \Rightarrow M)$ is trivial.
\end{proposition}

\begin{proof}
From Theorem \ref{theor:lin}, $H^\bullet_{\mathrm{def}, \mathrm{lin}} (TE)$ embeds into the deformation cohomology $H^\bullet_{\mathrm{def}} (TE)$ of the tangent algebroid $TE \Rightarrow E$ which is trivial (see, e.g., \cite{crainic:def}).
\end{proof}

Other applications of Theorem \ref{theor:lin} will be considered in Section \ref{Sec:examples}.

Remember from Subsection \ref{sec:lin_def} that a linear deformation cochain $c \in C^k_{\mathrm{def}, \mathrm{lin}}(W)$ is completely determined by its action on $k$ linear sections and on $k-1$ linear sections and a core section, and the action of its symbol on $k-1$ linear sections and on $k-2$ linear sections and a core section. We conclude this subsection providing a slightly more explicit description of the linearization map (\ref{eq:projs}) in terms of these restricted actions.

\begin{proposition} Let $c \in \mathfrak D^k (W;E)$. Then $c_{\mathrm{lin}}$ is completely determined by the following identities:
	\begin{enumerate}
		\item $c_{\mathrm{lin}}(\tilde \alpha_1, \dots, \tilde \alpha_k) = c(\tilde \alpha_1, \dots, \tilde \alpha_k)_{\mathrm{lin}}$,
		\item $c_{\mathrm{lin}}(\tilde \alpha_1, \dots, \tilde \alpha_{k-1}, \overline \chi) = c(\tilde \alpha_1, \dots, \tilde \alpha_{k-1}, \overline \chi)_{\mathrm{core}}$,
		\item $\sigma_{c_{\mathrm{lin}}}(\tilde \alpha_1, \dots, \tilde \alpha_{k-1}) = \sigma_c (\tilde \alpha_1, \dots, \tilde \alpha_{k-1})_{\mathrm{lin}}$,
		\item $\sigma_{c_{\mathrm{lin}}}(\tilde \alpha_1, \dots, \tilde \alpha_{k-2}, \overline \chi) = \sigma_c (\tilde \alpha_1, \dots, \tilde \alpha_{k-2}, \overline \chi)_{\mathrm{core}}$,
	\end{enumerate}
for all $\tilde \alpha_1, \dots, \tilde \alpha_k \in \Gamma_{\mathrm{lin}}(W,E)$, $\overline \chi \in \Gamma_{\mathrm{core}}(W,E)$.
\end{proposition}

\begin{proof} We first compute
\[
c_\mathrm{core}(\tilde \alpha_1, \dots, \tilde \alpha_k) = \lim_{\lambda \to 0} \lambda \cdot (h_\lambda^* c) (\tilde \alpha_1, \dots, \tilde \alpha_k) = \lim_{\lambda \to 0} \lambda \ h_\lambda^*(c(\tilde \alpha_1, \dots, \tilde \alpha_k)) = c(\tilde \alpha_1, \dots, \tilde \alpha_k)_{\mathrm{core}}.
\]
Then
\[
\begin{aligned}
c_\mathrm{lin}(\tilde \alpha_1, \dots, \tilde \alpha_k) & = \lim_{\lambda \to 0} \big( h_\lambda^* c - \lambda^{-1} \cdot c_\mathrm{core} \big) (\tilde \alpha_1, \dots, \tilde \alpha_k) \\ 
& = \lim_{\lambda \to 0} \left( h_\lambda^*(c(\tilde \alpha_1, \dots, \tilde \alpha_k)) - \lambda^{-1} c(\tilde \alpha_1, \dots, \tilde \alpha_k)_\mathrm{core} \right) = \\
& = c(\tilde \alpha_1, \dots, \tilde \alpha_k)_\mathrm{lin}.
\end{aligned}
\]
Identity (2) in the statement can be proved in a similar way. To prove (3) first notice that
\[
\sigma_{c_{\mathrm{core}}} = \sigma_{\lim_{\lambda \to 0} \lambda \cdot h_\lambda^\ast c} = \lim_{\lambda \to 0} \lambda \cdot \sigma_{h_\lambda^\ast c} = \lim_{\lambda \to 0} \lambda \cdot h_\lambda^\ast \sigma_{ c},
\]
where we used (\ref{eq:phi^*sigma}). Hence
\[
\begin{aligned}
\sigma_{c_{\mathrm{core}}} (\tilde \alpha_1, \ldots, \tilde \alpha_{k-1}) & =  \lim_{\lambda \to 0} (\lambda \cdot h_\lambda^\ast \sigma_{ c} )(\tilde \alpha_1, \ldots, \tilde \alpha_{k-1}) \\
& = \lim_{\lambda \to 0} (\lambda \cdot h_\lambda^\ast (\sigma_{ c} (\tilde \alpha_1, \ldots, \tilde \alpha_{k-1}))) \\
& = \sigma_{ c} (\tilde \alpha_1, \ldots, \tilde \alpha_{k-1})_{\mathrm{core}}.
\end{aligned}
\]
Similarly,
\[
\begin{aligned}
\sigma_{c_{\mathrm{lin}}} & = \sigma_{\lim_{\lambda \to 0} \left(h_\lambda^\ast c - \lambda^{-1} \cdot c_{\mathrm{core}} \right)} \\
& = \lim_{\lambda \to 0} \sigma_{h_\lambda^\ast c - \lambda^{-1} \cdot c_{\mathrm{core}}} \\
&= \lim_{\lambda \to 0} \left(h_\lambda^\ast\sigma_{ c } - \lambda^{-1} \sigma_{c_{\mathrm{core}}}\right),
\end{aligned}
\]
hence
\[
\begin{aligned}
\sigma_{c_{\mathrm{lin}}} (\tilde \alpha_1, \ldots, \tilde \alpha_{k-1}) & = \lim_{\lambda \to 0} \left((h_\lambda^\ast\sigma_{ c })(\tilde \alpha_1, \ldots, \tilde \alpha_{k-1}) - \lambda^{-1} \sigma_{c_{\mathrm{core}}}(\tilde \alpha_1, \ldots, \tilde \alpha_{k-1})\right) \\
& = \lim_{\lambda \to 0} \left(h_\lambda^\ast (\sigma_{ c }(\tilde \alpha_1, \ldots, \tilde \alpha_{k-1})) - \lambda^{-1} \sigma_{c}(\tilde \alpha_1, \ldots, \tilde \alpha_{k-1})_{\mathrm{core}}\right) \\
& = \sigma_{c} (\tilde \alpha_1, \ldots, \tilde \alpha_{k-1})_{\mathrm{lin}}
\end{aligned}.
\]
Identity (4) can be proved in a similar way. \end{proof}


\section{Examples and applications}\label{Sec:examples}
 
In this section we provide several examples. Examples in Subsections
\ref{sec:VB-alg}, \ref{sec:partial_conn} and \ref{sec:Lie_vect} parallel the analogous examples in \cite{crainic:def}, connecting our linear deformation cohomology to known cohomologies. Examples in 
Subsections \ref{sec:LA-vect}, \ref{sec:tangent-VB} and \ref{sec:type_1} are specific 
to VB-algebroids. 

\subsection{VB-algebras}\label{sec:VB-alg}

A \emph{VB-algebra} is a \emph{vector bundle object in the category of Lie algebras}. In other words, it is a VB-algebroid of the form
\[
\begin{array}{c}
\xymatrix{\mathfrak h \ar@{=>}[r] \ar[d] & \{0\} \ar[d] \\
 \mathfrak g \ar@{=>}[r]  & \{ \ast \}}
 \end{array}.
\]
In particular, $\mathfrak h$ and $\mathfrak g$ are Lie algebras. Now, let $C := \ker (\mathfrak h \to \mathfrak g)$ be the core of $(\mathfrak h \Rightarrow \{0\}; \mathfrak g \Rightarrow \{ \ast \})$. It easily follows from the definition of VB-algebroid that
\begin{itemize}
\item $C$ is a representation of $\mathfrak g$, and
\item $\mathfrak h = \mathfrak g \ltimes C$ is the semidirect product Lie algebra,
\item $\mathfrak h = \mathfrak g \ltimes C \to \mathfrak g$ is the projection onto the first factor.
\end{itemize}

Let $\operatorname{End} C$ denote endomorphisms of the vector space $C$. In the present case, the short exact sequence (\ref{eq:SESDGLAs}) reads
\begin{equation}\label{eq:seq_vb-alg}
0 \to C^\bullet(\mathfrak g, \operatorname{End} C) \to C^\bullet_{\mathrm{def,lin}}(\mathfrak h) \to C^\bullet_{\mathrm{def}}(\mathfrak g) \to 0,
\end{equation}
where $C^\bullet(\mathfrak g, \operatorname{End} C) = \wedge^{\bullet} \mathfrak g^* \otimes \operatorname{End} C$ is the Chevalley-Eilenberg complex of $\mathfrak g$ with coefficients in the induced representation $\operatorname{End} C$, and $C^\bullet_{\mathrm{def}}(\mathfrak g) = (\wedge^{\bullet} \mathfrak g^* \otimes \mathfrak g)[1]$ is the Chevalley-Eilenberg complex with coefficients in the adjoint representation. From the classical theory of Nijenhuis and Richardson \cite{nijenhuis:def_lie, nijenhuis:coh, nijenhuis:def}, the latter controls deformations of $\mathfrak g$, while the former controls deformations of the representation of $\mathfrak g$ on $C$.

The sequence (\ref{eq:seq_vb-alg}) has a natural splitting in the category of graded Lie algebras. Namely, there is an obvious graded Lie algebra map
\[
C^\bullet_{\mathrm{def}}(\mathfrak g) \to C^\bullet_{\mathrm{def,lin}}(\mathfrak h), \quad c \mapsto \tilde c
\]
given by
\[
\tilde c (v_1 + \chi_1, \ldots, v_{k+1} + \chi_{k+1}) := c (v_1, \ldots, v_{k+1})
\]
for all $c \in C^k_{\mathrm{def}}(\mathfrak g) = \wedge^{k+1} \mathfrak g^\ast \otimes \mathfrak g$, and all $v_i + \chi_i \in \mathfrak h = \mathfrak g \oplus C$, $i = 1, \ldots, k+1$. It is clear that the inclusion $C^\bullet_{\mathrm{def}}(\mathfrak g) \to C^\bullet_{\mathrm{def,lin}}(\mathfrak h)$ splits the projection $C^\bullet_{\mathrm{def,lin}}(\mathfrak h) \to C^\bullet_{\mathrm{def}}(\mathfrak g)$. Hence 
\begin{equation}\label{eq:split_VB_alg}
C^\bullet_{\mathrm{def,lin}}(\mathfrak h) \cong C^\bullet(\mathfrak g, \operatorname{End} C) \oplus C^\bullet_{\mathrm{def}}(\mathfrak g).
\end{equation}
as graded Lie algebras. However (\ref{eq:split_VB_alg}) is not a DGLA isomorphism. We now describe the differential $\delta$ in $C^\bullet_{\mathrm{def,lin}}(\mathfrak h)$ in terms of the splitting
(\ref{eq:split_VB_alg}). First of all, denote by $\theta : \mathfrak g \to \operatorname{\mathfrak{End}} C$ the action of $\mathfrak g$ on $C$, and let
\[
\Theta : \wedge^\bullet \mathfrak g^\ast \otimes \mathfrak g \to \wedge^\bullet \mathfrak g^\ast \otimes \operatorname{End} C,
\]
be the map obtained from $\theta$ by extension of scalars. From the properties of the action, $\Theta$ is actually a cochain map
\[
\Theta : C^\bullet_{\mathrm{def}}(\mathfrak g)[-1] \to  C^\bullet(\mathfrak g, \operatorname{End} C).
\]
Finally, a direct computation, reveals that the isomorphism (\ref{eq:split_VB_alg}) identifies the differential in 
$
C^\bullet_{\mathrm{def,lin}}(\mathfrak h)
$ 
with that of the mapping cone
$
\operatorname{Cone} (\Theta)
$:
\[
(C^\bullet_{\mathrm{def,lin}}(\mathfrak h), \delta) \cong \operatorname{Cone} (\Theta)
\]
as cochain complexes. Notice that the long exact cohomology sequence of the mapping cone is just (\ref{eq:LES_VB}).

\subsection{LA-vector spaces}\label{sec:LA-vect}
An \emph{LA vector space} is a \emph{Lie algebroid object in the category of vector spaces}. In other words, it is a VB-algebroid of the form
\[
\begin{array}{c}
\xymatrix{W \ar@{=>}[r] \ar[d] & E \ar[d] \\
\{0\} \ar@{=>}[r]  & \{ \ast \}}
 \end{array}.
\]
In particular, $W$ and $E$ are vector spaces. Now, let $C := \ker (W \to E)$ be the core of $(W \Rightarrow E; \{0\} \Rightarrow \{\ast\})$. It easily follows from the definition of VB-algebroid that $W$ identifies canonically with the direct sum $C \oplus E$ and all the structure maps are completely determined by a linear map $\partial : C \to E$. Specifically, sections of $W \to E$ are the same as smooth maps $E \to C$, and given a basis $(C_I)$ of $C$, the Lie bracket on maps $E \to C$ is given by
\begin{equation}\label{eq:LA-vect1}
[f, g] = f^I (\partial C_I)^\uparrow g - g^I (\partial C_I)^\uparrow f, 
\end{equation}
where $f = f^I C_I$ and $g = g^I C_I$. It follows that the anchor $\rho : W \to TE$ is given on sections by
\begin{equation}\label{eq:LA-vect2}
\rho (f) = f^I (\partial C_I)^\uparrow.
\end{equation}

Linear deformations of $(W \Rightarrow E; \{0\} \Rightarrow \{\ast\})$ are the same as deformations of $\partial$ as a linear map. Let us describe the linear deformation complex explicitly. 
As the bottom Lie algebroid is trivial, $C^\bullet_{\mathrm{def}, \mathrm{lin}} (W)$ consists of graded endomorphisms $\operatorname{End} (C[1] \oplus E)^{\bullet}$ of the graded vector space $W[1]_E = C[1] \oplus E$. From  (\ref{eq:LA-vect1}) and  (\ref{eq:LA-vect2}) the differential $\delta$ in $\operatorname{End} (C[1] \oplus E)^{\bullet}$ is just the commutator with $\partial$, meaning that the deformation cohomology consists of homotopy classes of graded cochain maps $(C[1] \oplus E, \partial) \to (C[1] \oplus E, \partial)$. More explicitly, $(\operatorname{End} (C[1] \oplus E)^{\bullet}, \delta)$ is concentrated in degrees $-1, 0, 1$. Namely, it is
\[
0 \longrightarrow \operatorname{Hom} (E, C)[1] \overset{\delta_0}{\longrightarrow}  \operatorname{End} (C) \oplus \operatorname{End}  (E)  \overset{\delta_1}{\longrightarrow}  \operatorname{Hom} (C, E)[-1] \longrightarrow 0,
\]
where $\delta_0$ and $\delta_1$ are given by:
\[
\begin{aligned}
\delta_0 \phi & = (\phi \circ \partial, \partial \circ \phi), \\
\delta_1 (\psi_C, \psi_E) & = \partial \circ \psi_C - \psi_E \circ \partial,
\end{aligned}
\]
where $\phi \in \operatorname{Hom} (E, C)$, $\psi_C \in \operatorname{End} (C)$ and $\psi_E \in \operatorname{End} (E)$. We conclude immediately that 
\[
H^\bullet_{\mathrm{def}, \mathrm{lin}} (W) = \operatorname{End} (\operatorname{coker} \partial \oplus \ker \partial [1])^{\bullet}
\]
i.e.
\[
\begin{aligned}
H^{-1}_{\mathrm{def}, \mathrm{lin}} (W) & = \operatorname{Hom} (\operatorname{coker} \partial, \ker \partial),\\
H^0_{\mathrm{def}, \mathrm{lin}} (W) & = \operatorname{End} (\operatorname{coker} \partial) \oplus \operatorname{End} (\ker \partial), \\
H^1_{\mathrm{def}, \mathrm{lin}} (W) & = \operatorname{Hom}(\ker \partial, \operatorname{coker} \partial). 
\end{aligned}
\]
This shows, for instance, that infinitesimal deformations of a linear map $\partial : C \to E$ are all trivial if and only if $\partial$ is injective or surjective, as expected.
 
\subsection{Tangent and cotangent VB-algebroids}\label{sec:tangent-VB}

Let $A \Rightarrow M$ be a Lie algebroid.   Then $(TA \Rightarrow TM; A \Rightarrow M)$ is a VB-algebroid, called the \emph{tangent VB-algebroid} of $A$. The structure maps of the Lie algebroid $TA \Rightarrow TM$ are defined as follows. First of all recall that $(TA \to TM; A \to M)$ is a DVB whose core is canonically isomorphic to $A$ itself. In particular, any section $\alpha$ of $A$ determines a core section $\overline \alpha$ of $TA \to TM$. A section $\alpha$ of $A$ does also determine a linear section $T \alpha$ of $TA \to TM$: its tangent map.

Denote by $\tau : TM \to M$ the projection. In the following, for a vector field $X \in \mathfrak X (M)$, we denote by $X_{\mathrm{tan}} \in \mathfrak X (TM)$ its \emph{tangent lift}. By definition, the flow of $X_{\mathrm{tan}}$ is obtained from the flow of $X$ by taking the tangent diffeomorphisms. Equivalently, $X_{\mathrm{tan}}$ is the (linear) vector field on $TM$ uniquely determined by
\begin{equation}\label{eq:tan_lift}
X_{\mathrm{tan}}(\ell_{df}) = \ell_{d X(f)}, \quad \text{and} \quad X_{\mathrm{tan}}(\tau^\ast) = \tau^\ast X(f),
\end{equation}
for all $f \in C^\infty (M)$. Here $\ell_{df}$ is the fiber-wise linear function on $TM$ corresponding to the $1$-form $df$ (viewed as a section of the dual bundle $T^\ast M$). Notice that Formulas (\ref{eq:tan_lift}) can be used to defined the tangent lift of vector fields on a graded manifold. This will be indeed useful below.

Now we come back to the tangent VB-algebroid $(TA \Rightarrow TM; A \Rightarrow M)$. The anchor $\rho_{TA} : TA \to TTM$ is determined by
\begin{equation}\label{eq:anchor_TA}
\rho_{TA}(T\alpha) = \rho(\alpha)_{\operatorname{tan}}, \quad \rho_{TA}(\overline \alpha) = \rho(\alpha)^{\uparrow}.
\end{equation}
and the bracket $[-,-]_{TA}$ in $\Gamma (TA, TM)$ is completely determined by:
\begin{equation}\label{eq:bracket_TA}
[T\alpha, T\beta]_{TA} = T[\alpha,\beta], \quad  [T\alpha, \overline \beta]_{TA} = \overline{[\alpha, \beta]}, \quad [\overline \alpha, \overline \beta]_{TA} = 0,
\end{equation}
for all $\alpha, \beta \in \Gamma(A)$. 
The dual VB-algebroid $(T^\ast A \Rightarrow A^\ast; A \Rightarrow M)$ of the tangent VB-algebroid is called the \emph{cotangent VB-algebroid}. We want to discuss the linear deformation cohomology of $(TA \Rightarrow TM; A \Rightarrow M)$ (hence of $(T^\ast A \Rightarrow A^\ast; A \Rightarrow M)$). We use the graded geometric description. Deformation cochains of $TA \Rightarrow TM$ are vector fields on the graded manifold $TA [1]_{TM}$ obtained from $TA$ shifting by one the degree in the fibers of the vector bundle $TA \to TM$. Linear deformation cochains are vector fields that are linear with respect to the vector bundle structure $TA[1]_{TM} \to A[1]$. 
\begin{lemma}\label{lem:iota}
	Let $T A[1]$ be the tangent bundle of $A[1]$ and let $\tau : TA[1] \to A[1]$ be the projection. There is a canonical isomorphism of vector bundles of graded manifolds
	\[
	\xymatrix{ TA [1]_{TM} \ar[rr]^\iota \ar[dr]& & TA[1] \ar[dl] \\
		& A[1] &}
	\]
	uniquely determined by the following condition: 
	\begin{equation}\label{eq:iota}
	\langle \iota^\ast \ell_{d \omega}, T\alpha_1 \wedge \cdots \wedge T\alpha_k \rangle = \ell_{d \langle \omega , \alpha_1 \wedge \cdots \wedge \alpha_k \rangle}
	\end{equation}
	for all $\omega \in C^\infty (A[1])^\bullet = C^\bullet (A)$ of degree $k$, all sections $\alpha_1, \ldots, \alpha_k \in \Gamma (A)$, and all $k$. Additionally
	\begin{equation}\label{eq:iota_2}
	\langle \iota^\ast \ell_{d \omega}, T\alpha_1 \wedge \cdots \wedge T\alpha_{k-1} \wedge \overline \alpha_k \rangle = \tau^\ast \langle \omega , \alpha_1 \wedge \cdots \wedge \alpha_k \rangle.
	\end{equation}
\end{lemma}
Formulas (\ref{eq:iota}) and (\ref{eq:iota_2}) in the statement require some explanations. The expression $d \omega$ in the left hand side should be interpreted as a $1$-form on $A[1]$, the de Rham differential of the function $\omega$, and $\ell_{d \omega}$ is the associated fiber-wise linear function on $T A[1]$. The pull-back of $\ell_{d\omega}$ along $\iota^\ast$ is a function on $TA[1]_{TM}$, i.e.~a $C^\infty (TM)$-valued, skew-symmetric multilinear map on sections of $TA \to TM$. The $T\alpha$ are the tangent maps $T\alpha : TM \to TA$ of the $\alpha : M \to A$. In particular they are linear sections of $TA \to TM$. The right hand side of (\ref{eq:iota}) is the fiber-wise linear function on $TM$ corresponding to the $1$-form $d \langle \omega , \alpha_1 \wedge \cdots \wedge \alpha_k \rangle$ on $M$. Here we interpret $\omega$ as a skew-symmetric multilinear map on sections of $A$.

\begin{proof}[Proof of Lemma \ref{lem:iota}]
	Let $(x^i)$ be coordinates on $M$, let $(u_\alpha)$ be a local basis of $\Gamma (A)$, and let $(u^\alpha)$ be the associated fiber-wise linear coordinates on $A$. These data determine coordinates $(x^i, \tilde u{}^\alpha)$ on $A[1]$ in the obvious way. In particular the $x^i$ have degree $0$ and the $\tilde u{}^\alpha$ have degree $1$. We also consider standard coordinates $(x^i, u^\alpha, \dot{x}{}^i, \dot{u}{}^\alpha)$ induced by $(x^i, u^\alpha)$ on $TA$. Notice that $(u^\alpha, \dot{u}{}^\alpha)$ are fiber-wise linear coordinates with respect to the vector bundle structure $TA \to TM$. More precisely, they are the fiber-wise linear coordinates associated to the local basis $(T u_\alpha, \overline u_\alpha)$ of $\Gamma (TA, TM)$. Next we denote by $(x^i, \dot{x}{}^i, \tilde u{}^\alpha, \tilde{\dot{u}}{}^\alpha)$ the induced coordinates on $TA[1]_{TM}$. They have degree $0,0,1,1$ respectively. Finally, we denote $(x^i, \tilde u{}^\alpha, X^i, \tilde U{}^\alpha)$ the standard coordinates on $TA[1]$ induced by $(x^i, \tilde u{}^\alpha)$. Define $\iota$ by putting
	\[
	\iota^\ast X^i = \dot{x}{}^i, \quad \text{and} \quad \iota^\ast \tilde U{}^\alpha = \tilde{\dot{u}}{}^\alpha.
	\]
	A direct computation exploiting the appropriate transition maps reveals that $\iota$ is globally well defined. Now we prove (\ref{eq:iota}). We work in coordinates. Take a degree $k$ function $\omega = f_{\alpha_1 \cdots \alpha_k} (x) \tilde u{}^{\alpha_1} \cdots \tilde u{}^{\alpha_k}$ on $A[1]$. A direct computation shows that
	\begin{equation}\label{eq:iota_l}
	\iota^\ast \ell_{d \omega} = \frac{\partial f_{\alpha_1 \cdots \alpha_k} }{\partial x^i}  \tilde u{}^{\alpha_1} \cdots \tilde u{}^{\alpha_k} \dot{x}{}^i + k f_{\alpha_1 \cdots \alpha_k} \tilde u{}^{\alpha_1} \cdots  \tilde u{}^{\alpha_{k-1}} \tilde{\dot{u}}{}^{\alpha_k}.
	\end{equation}
	Now, let $\alpha_1, \ldots, \alpha_k \in \Gamma (A)$, and $a = 1, \ldots, k$. If $\alpha_a$ is locally given by $\alpha_a = g_a^\alpha (x) u_\alpha$, then
	\[
	T\alpha_a = \frac{\partial g^\alpha_a}{\partial x^i} \dot{x}{}^i \overline u{}_\alpha + g^\alpha_a Tu_\alpha,
	\]
	and, from (\ref{eq:iota_l}),
	\[
	\begin{aligned}
	\langle \iota^\ast \ell_{d \omega}, T \alpha_1 \wedge \cdots \wedge T \alpha_k \rangle 
	&=  k! \left( \frac{\partial f_{\alpha_1 \cdots \alpha_k} }{\partial x^i}  g_1^{\alpha_1} \cdots g_k^{\alpha_k}+  f_{\alpha_1 \cdots \alpha_k} g_1^{\alpha_1} \cdots  g_{k-1}^{\alpha_{k-1}} \frac{\partial g^{\alpha_k}_k}{\partial x^i} \right)\dot{x}{}^i \\
	& = k! \frac{\partial}{\partial x^i} \left( f_{\alpha_1 \cdots \alpha_k} g_1^{\alpha_1} \cdots g_k^{\alpha_k} \right) \dot{x}{}^i  = \ell_{d \langle \omega , \alpha_1 \wedge \cdots \wedge \alpha_k \rangle} .
	\end{aligned}
	\]
	Identity (\ref{eq:iota_2}) is proved in a similar way. To see that there is no other vector bundle isomorphism $\iota : TA[1]_{TM} \to TA[1]$ with the same property
	(\ref{eq:iota}) notice that $X^i = \ell_{dx^i}$ and $\tilde U{}^\alpha = \ell_{d\tilde u^\alpha}$. Now use (\ref{eq:iota}) to show that $\iota^\ast X^i = \dot{x}{}^i$ and $\iota^\ast \tilde U{}^\alpha = \tilde{\dot{u}}{}^\alpha$, necessarily.
\end{proof}

In the following we will understand the isomorphism $\iota$ of Lemma \ref{lem:iota}, and identify $TA[1]_{TM}$ with $TA[1]$. Now, recall that $C^\bullet_{\mathrm{def}, \mathrm{lin}} (TA) = \mathfrak X_{\mathrm{lin}} (TA[1])^{\bullet}$ fits in the short exact sequence of DGLAs:
\begin{equation}\label{eq:SES_TA}
0 \to \operatorname{\mathfrak{End}}( TA[1])^\bullet \to \mathfrak X_{\mathrm{lin}} (TA[1])^{\bullet} \to \mathfrak X (A[1])^{\bullet} \to 0.
\end{equation}
The tangent lift 
\begin{equation}\label{eq:tan}
\operatorname{tan} : \mathfrak X (A[1])^{\bullet} \hookrightarrow \mathfrak X (TA[1])^{\bullet}, \quad X \mapsto X_{\mathrm{tan}}
\end{equation}
splits the sequence (\ref{eq:SES_TA}) in the category of DGLAs. As $\mathfrak X (A[1])^{\bullet} = C^\bullet_{\mathrm{def}}(A)$, we immediately have the following
\begin{proposition}
	For every Lie algebroid $A \Rightarrow M$ there is a direct sum decomposition
	\[
	H^\bullet_{\mathrm{def}, \mathrm{lin}} (TA) = H^\bullet_{\mathrm{def}, \mathrm{lin}} (T^\ast A) = H^\bullet (\operatorname{\mathfrak{End}} (TA[1])) \oplus H^\bullet_{\mathrm{def}} (A).
	\]
\end{proposition}
In the last part of the subsection we describe the inclusion (\ref{eq:tan}) in terms of deformation cochains. This generalizes (\ref{eq:anchor_TA}) and (\ref{eq:bracket_TA}) to possibly higher cochains. Using the canonical isomorphisms $C^\bullet_{\mathrm{def}, \mathrm{lin}} (TA) = \mathfrak X_{\mathrm{lin}} (TA[1])^{\bullet}$, and $C^\bullet_{\mathrm{def}} (A) = \mathfrak X (A[1])^{\bullet}$ we get an inclusion
\[
\operatorname{tan} : C^\bullet_{\mathrm{def}} (A) \hookrightarrow C^\bullet_{\mathrm{def}, \mathrm{lin}} (TA), \quad c \mapsto c_{\mathrm{tan}}.
\]

\begin{proposition}\label{prop:tan_VB_class}
 Let $c \in C^{k-1}_{\mathrm{def}}(A)$. Then $c_{\operatorname{tan}} \in C^{k-1}_{\mathrm{def}, \mathrm{lin}}(TA)$ satisfies:
	\begin{enumerate}
		\item $c_{\operatorname{tan}}(T\alpha_1, \dots, T\alpha_k) = T c(\alpha_1, \dots, \alpha_k)$,
		\item $c_{\operatorname{tan}}(T\alpha_1, \dots, T\alpha_{k-1}, \overline \alpha_k) = \overline{c(\alpha_1, \dots, \alpha_k)}$,
		\item $\sigma_{c_{\operatorname{tan}}} (T\alpha_1, \dots, T\alpha_{k-1}) = \sigma_c (\alpha_1, \dots, \alpha_{k-1})_{\mathrm{tan}}$,
		\item $\sigma_{c_{\operatorname{tan}}} (T\alpha_1, \dots, T\alpha_{k-2}, \overline \alpha_{k-1})  = \sigma_c (\alpha_1, \dots, \alpha_{k-1})^\uparrow$,
	\end{enumerate}
	for all $\alpha_1, \dots, \alpha_k \in \Gamma(A)$. Identities (1)-(4) (together with the fact that $c_{\mathrm{tan}}$ is a linear cochain) determine $c_{\mathrm{tan}}$ completely.
\end{proposition}

\begin{proof} 
	We begin with (3). Recall that the tangent lift $X_{\mathrm{tan}}$ of a vector field $X$ is completely determined by (\ref{eq:tan_lift}) (and this remains true in the graded setting). So, let $X \in \mathfrak{X}(A[1])^\bullet$ be the graded vector field corresponding to $c$ (hence $X_{\mathrm{tan}} \in \mathfrak X (TA[1]_{TM})$ is the graded vector field corresponding to $c_{\mathrm{tan}}$), let $f \in C^\infty (M)$ and let $\alpha_1, \ldots, \alpha_{k-1} \in \Gamma (A)$. Using (\ref{eq:delta_c_f}), compute
	\[
	\sigma_{c_{\operatorname{tan}}}(T\alpha_1, \dots, T\alpha_{k-1}) \ell_{df} = \langle X_{\operatorname{tan}}(\ell_{df}), T\alpha_1 \wedge \cdots \wedge T\alpha_{k-1} \rangle = \langle \ell_{d (X(f))}, T\alpha_1 \wedge \cdots \wedge T\alpha_{k-1}\rangle.
	\]
	From (\ref{eq:iota}),
	\[
	\langle \ell_{d (X(f))}, T\alpha_1 \wedge \cdots \wedge T\alpha_{k-1}\rangle = \ell_{d\langle X(f), \alpha_1\wedge \cdots \wedge \alpha_{k-1}\rangle} = \ell_{d (\sigma_c (\alpha_1, \ldots, \alpha_{k-1}) f)} =\sigma_c (\alpha_1, \ldots, \alpha_{k-1})_{\mathrm{tan}} \ell_{df}.
	\]
	As $\sigma_{c_{\operatorname{tan}}}(T\alpha_1, \dots, T\alpha_{k-1}) $ and $\sigma_c (\alpha_1, \ldots, \alpha_{k-1})_{\mathrm{tan}} $ are both linear and they both project onto $\sigma_c (\alpha_1, \ldots, \alpha_{k-1})$, this is enough to conclude that $\sigma_{c_{\operatorname{tan}}}(T\alpha_1, \dots, T\alpha_{k-1}) = \sigma_c (\alpha_1, \ldots,\alpha_{k-1})_{\mathrm{tan}} $. Identity (4) can be proved in a similar way from (\ref{eq:iota_2}) using that
	\[
	\sigma_c (\alpha_1, \ldots, \alpha_{k-1})^\uparrow \ell_{df} = \tau^\ast \langle df, \sigma_c (\alpha_1, \ldots, \alpha_{k-1}) \rangle = \tau^\ast (\sigma_c (\alpha_1, \ldots, \alpha_{k-1}) f). 
	\]

	We now prove (1). Both sides of the identity are linear sections of $TA \to TM$ and one can easily check in local coordinates that a linear section $\tilde \alpha$ is completely determined by pairings of the form $\langle \ell_{d\varphi}, \tilde \alpha \rangle$. Here, $\varphi$ is a section of $A^\ast \to M$ seen as a degree 1 function on $A[1]$, $d\varphi$ is its de Rham differential, and $\ell_{d\varphi}$ is the associated degree 1 fiber-wise linear function on $TA[1]$, which, in turn, can be interpreted as a $1$-form on the algebroid $TA \Rightarrow TM$, as in Lemma (\ref{lem:iota}).
	
	So, take $\varphi \in \Gamma(A^\ast)$, $c \in C^{k-1}_{\operatorname{def}}(A)$, $\alpha_1, \dots, \alpha_k \in \Gamma(A)$, and compute
	\[
	\begin{aligned}
	& \langle \ell_{d\varphi}, c_{\operatorname{tan}}(T\alpha_1, \dots, T\alpha_k) \rangle \\
	& = \langle X_{\operatorname{tot}} (\ell_{d\varphi}), T \alpha_1 \wedge \dots \wedge T \alpha_k \rangle  - \sum_i (-)^{k-i} \sigma_{c_{\operatorname{tan}}} (T \alpha_1, \dots, \widehat{T\alpha_i}, \dots, T\alpha_k) \langle \ell_{d \varphi}, T \alpha_i \rangle \\
	& = \langle \ell_{d(X(\varphi))}, T\alpha_1 \wedge \dots \wedge T\alpha_k \rangle - \sum_i (-)^{k-i} \sigma_c (\alpha_1, \dots, \widehat{\alpha_i}, \dots, \alpha_k)_{\operatorname{tan}} \ell_{d \langle \omega, \alpha_i \rangle} \\
	& = \ell_{d \langle X(\varphi), \alpha_1 \wedge \dots \wedge \alpha_k \rangle} - \sum_i (-)^{k-i} \ell_{d( \sigma_c (\alpha_1, \dots, \widehat{\alpha_i}, \dots, \alpha_k) \langle \omega, \alpha_i \rangle)} \\
	& = \ell_{d \langle \varphi, c(\alpha_1, \dots, \alpha_k) \rangle} = \langle \ell_{d \varphi}, Tc(\alpha_1, \dots, \alpha_k) \rangle,
	\end{aligned}
	\]
	where we used, in particular, (\ref{eq:def_der}),  the first one in (\ref{eq:tan_lift}), Identity (3), and (\ref{eq:iota}). So (1) holds.
	
	Identity (2) can be proved in a similar way using (\ref{eq:def_der}), both identities (\ref{eq:tan_lift}), Identity (4), and (\ref{eq:iota_2}). We leave to the reader the straightforward details.
\end{proof}

\begin{remark}
	Proposition \ref{prop:tan_VB_class} shows, in particular, that the Lie bracket $b_A$ on $\Gamma (A)$ and the Lie bracket $b_{TA}$ in $\Gamma (TA, TM)$ are related by $b_{TA} = (b_A)_{\mathrm{tan}}$.
\end{remark}

\subsection{Partial connections}\label{sec:partial_conn}

Let $M$ be a manifold, $D \subset TM$ an involutive distribution, and let $\mathcal F$ be the integral foliation of $D$. In particular $D \Rightarrow M$ is a Lie algebroid with injective anchor. A flat (partial) $D$-connection $\nabla$ in a vector bundle $E \to M$ defines a VB-algebroid 
\[
\begin{array}{c}
\xymatrix{ H \ar[d] \ar@{=>}[r] & E \ar[d] \\
	D \ar@{=>}[r] & M}
\end{array},
\]
where $H \subset TE$ is the \emph{horizontal distribution} determined by $D$. Notice that the core of $(H \Rightarrow E; D \Rightarrow M)$ is trivial, and every VB-algebroid with injective (base) anchor and trivial core arises in this way. Hence, (small) deformations of $(H \Rightarrow E; D \Rightarrow M)$ are the same as simultaneous deformations of the foliation $\mathcal F$ and the flat partial connection $\nabla$. We now discuss the linear deformation cohomology. Denote by $q : E \to M$ the projection. First of all, the de Rham complex of $D \Rightarrow M$ is the same as leaf-wise differential forms $\Omega^\bullet (\mathcal F)$ with the leaf-wise de Rham differential $d_{\mathcal F}$. Hence, the deformation complex of $D$ consists of derivations of $\Omega^\bullet (\mathcal F)$ (the differential being the graded commutator with $d_{\mathcal F}$). As the core of $(H \Rightarrow E; D \Rightarrow M)$ is trivial, there is a canonical isomorphism $H \cong q^\ast D$ of vector bundles over $E$. It easily follows that
 the linear deformation complex $(C^\bullet_{\mathrm{def}, \mathrm{lin}}(H), \delta)$ consists of derivations of the graded module $\Omega^{\bullet} (\mathcal F, E)$ of $E$-valued, leaf-wise differential forms, and the differential $\delta$ is the commutator with the (leaf-wise partial) connection differential $d^\nabla_{\mathcal F}$. The kernel of $C^\bullet_{\mathrm{def}, \mathrm{lin}}(H) \to C^\bullet_{\mathrm{def}} (D)$ consists of graded $\Omega^\bullet(\mathcal F)$-linear endomorphisms of $\Omega^\bullet (\mathcal F, E)$. The latter are the same as $\operatorname{End} E$-valued leaf-wise differential forms $\Omega^\bullet (\mathcal F, \operatorname{End} E)$, and the restricted differential is the connection differential (corresponding to the induced connection in $\operatorname{End} E$).

Now, denote by $\nu = TM / D$ the normal bundle to $\mathcal F$. It is canonically equipped with the Bott connection $\nabla^{\mathrm{Bott}}$, and there is a deformation retraction, hence a quasi-isomorphism, $\pi : C^\bullet_{\mathrm{def}}(D) \to \Omega^\bullet (\mathcal F, \nu)$ that maps a deformation cochain $c$ to the composition $\pi (c)$ of the symbol $\sigma_c : \wedge^\bullet D \to TM$ followed by the projection $TM \to \nu$. A similar construction can be applied to linear deformation cochains. To see this, first notice that derivations of $E$ modulo covariant derivatives along $\nabla$, $\mathfrak D (E) / \operatorname{im} \nabla$, are sections of a vector bundle $\tilde \nu \to M$. Additionally, $\tilde \nu$ is canonically equipped with a flat partial connection, also called the \emph{Bott connection} and denoted $\nabla^{\mathrm{Bott}}$,  defined by
\[
\nabla^{\mathrm{Bott}}_{X} (\Delta \operatorname{mod} \operatorname{im} \nabla) = [\nabla_X, \Delta] \operatorname{mod} \operatorname{im} \nabla
\] 
for all $\Delta \in \mathfrak D (E)$, and $X \in \Gamma (D)$. The symbol map $\sigma : \mathfrak D (E) \to \mathfrak X (M)$ descends to a surjective vector bundle map $\tilde \nu \to \nu$, intertwining the Bott connections. As $\operatorname{End} E \cap \operatorname{im} \nabla = 0$, we have $\ker (\tilde \nu \to \nu) = \operatorname{End} E $. In other words, there is a short exact sequence of vector bundles with partial connections:
\[
0 \longrightarrow \operatorname{End} E  \longrightarrow \tilde \nu  \longrightarrow \nu  \longrightarrow 0.
\]

Now, we define a surjective cochain map $\tilde \pi : C^\bullet_{\mathrm{def}, \mathrm{lin}}(H) \to \Omega^\bullet (\mathcal F, \tilde \nu)$. Let $\tilde c$ be a linear deformation cochain. Its symbol $\sigma_{\tilde c}$ maps linear sections of $H \to E$ to linear vector field on $E$. As $H \cong q^\ast D$, linear sections identify with plain sections of $D$. Accordingly $\sigma_{\tilde c}$ can be seen as a $\mathfrak D (E)$-valued $D$-form. Take this point of view and denote by $\tilde \pi (\tilde c) : \wedge^\bullet D \to \tilde \nu$ the composition of $\sigma_{\tilde c}$ followed by the projection $\mathfrak D (E) \to \Gamma (\tilde \nu)$.  

Summarizing, we have the following commutative diagram
\[
\begin{array}{c}
\xymatrix{
	0 \ar[r] & \Omega^\bullet (\mathcal F, \operatorname{End} E) \ar[r] \ar@{=}[d] & C^\bullet_{\mathrm{def}, \mathrm{lin}} (H) \ar[r] \ar[d]^{\tilde \pi} & C^\bullet_{\mathrm{def}} (D) \ar[r]  \ar[d]^{\pi}  & 0 \\
	0 \ar[r] & \Omega^\bullet (\mathcal F,  \operatorname{End} E) \ar[r] & \Omega^\bullet (\mathcal F,  \tilde \nu ) \ar[r] & \Omega^\bullet (\mathcal F,  \nu ) \ar[r] & 0
}
\end{array}.
\]
The rows are short exact sequences of DG-modules, and the vertical arrows are DG-module surjections. Additionally, $\pi$ is a quasi-isomorphism. Hence, it immediately follows from the Snake Lemma and the Five Lemma that $\tilde \pi$ is a quasi-isomorphism as well. We have thus proved the following

\begin{proposition}
	There is a canonical isomorphism of graded vector spaces between the linear deformation cohomology of the VB-algebroid $(H \Rightarrow E; D \Rightarrow M)$, and the leaf-wise cohomology with coefficients in $\tilde \nu$:
	\[
	H^\bullet_{\mathrm{def}, \mathrm{lin}} (H) = H^\bullet (\mathcal F, \tilde \nu).
	\]
\end{proposition}

\subsection{Lie algebra actions on vector bundles}\label{sec:Lie_vect}

Let $\mathfrak g$ be a (finite dimensional, real) Lie algebra acting on a vector bundle $E \to M$ by infinitesimal vector bundle automorphisms. In particular $\mathfrak g$ acts on $M$ and there is an associated action Lie algebroid $\mathfrak g \ltimes M \Rightarrow M$. Additionally, $\mathfrak g$ acts on the total space $E$ by linear vector fields. Equivalently, there is a Lie algebra homomorphism $\mathfrak g \to \mathfrak D (E)$ covering the (infinitesimal) action $\mathfrak g \to \mathfrak X (M)$. It follows that $(\mathfrak g \ltimes E \Rightarrow E; \mathfrak g \ltimes M \Rightarrow M)$ is a VB-algebroid. We want to discuss linear deformation cohomologies of $\mathfrak g \ltimes E \Rightarrow E$. We begin reviewing Crainic and Moerdijk remarks on the deformation cohomology of $\mathfrak g \ltimes M \Rightarrow M$ \cite{crainic:def} providing a graded geometric interpretation. The deformation complex $C^\bullet_{\mathrm{def}}(\mathfrak g \ltimes M)$ consists of vector fields on $(\mathfrak g \times M)[1] = \mathfrak g[1] \times M$. Denote by $\pi_{\mathfrak g}: \mathfrak g[1] \times M \to \mathfrak g [1]$ the projection. Composition on the right with the pull-back
\[
\pi_{\mathfrak g}^\ast : C^\infty (\mathfrak g [1])^\bullet \to C^\infty (\mathfrak g[1] \times M)^\bullet
\]
establishes a projection from vector fields on $\mathfrak g[1] \times M$ to $\pi_{\mathfrak g}$-relative vector fields $\mathfrak X_{\mathrm{rel}} (\pi_{\mathfrak g})^\bullet$, i.e.~vector fields on $\mathfrak g[1]$ with coefficients in functions on $\mathfrak g [1] \times M$:
\begin{equation}\label{eq:proj}
\mathfrak X ( \mathfrak g[1] \times M)^\bullet \to \mathfrak X_{\mathrm{rel}} (\pi_{\mathfrak g})^\bullet, \quad X \mapsto X \circ \pi_{\mathfrak g}^\ast.
\end{equation}
The kernel of projection (\ref{eq:proj})  consists of $\pi_\mathfrak g$-vertical vector fields $\mathfrak X^{\pi_\mathfrak g} (\mathfrak g[1] \times M)^\bullet$. Denote by $d_{\mathfrak g} \in \mathfrak X (\mathfrak g[1] \times M)^\bullet$ the homological vector field on $\mathfrak g[1] \times M$. The graded commutator $\delta := [d_{\mathfrak g}, -]$ preserves $\pi_\mathfrak g$-vertical vector fields. Hence there is a short exact sequence of cochain complexes:
\begin{equation}\label{SES:gM}
0 \to \mathfrak X^{\pi_\mathfrak g} (\mathfrak g[1] \times M)^\bullet \to \mathfrak X (\mathfrak g[1] \times M)^\bullet \to \mathfrak X_{\mathrm{rel}} (\pi_{\mathfrak g})^\bullet \to 0.
\end{equation}
Now, $\mathfrak X (\mathfrak g[1] \times M)^\bullet$ is exactly the deformation complex of $\mathfrak g \ltimes M$. Similarly, $\mathfrak X_{\mathrm{rel}} (\pi_{\mathfrak g})^\bullet$ is (canonically isomorphic to) the Chevalley-Eilenberg cochain complex of $\mathfrak g$ with coefficients in $C^\infty (M) \otimes \mathfrak g$: the tensor product of $C^\infty (M)$ and the adjoint representation. Following \cite{crainic:def}, we shortly denote this tensor product by $\mathfrak g_M$. Finally, $\mathfrak X^{\pi_\mathfrak g} (\mathfrak g[1] \times M)^\bullet$ is canonically isomorphic to the Chevalley-Eilenberg cochain complex of $\mathfrak g$ with coefficients in $\mathfrak X (M)$, up to a shit by $1$. So there is a short exact sequence of cochain complexes
\[
0 \to C^\bullet (\mathfrak g, \mathfrak X (M)) \to C^\bullet_{\mathrm{def}}(\mathfrak g \ltimes M) \to C^{\bullet + 1} (\mathfrak g, \mathfrak g_M) \to 0,
\]
and a long exact cohomology sequence
\begin{equation}\label{eq:LES}
\cdots \to H^k (\mathfrak g, \mathfrak X (M)) \to H^k_{\mathrm{def}}(\mathfrak g \ltimes M) \to H^{k+1} (\mathfrak g, \mathfrak g_M) \to  H^{k+1}(\mathfrak g, \mathfrak X (M)) \to \cdots.
\end{equation}
 
 We now pass to $\mathfrak g \ltimes E$. The linear deformation complex $C^\bullet_{\mathrm{def}, \mathrm{lin}} (\mathfrak g \ltimes E)$ consists of linear vector fields on $\mathfrak g [1] \times E$. Similarly as above, we consider the projection $\tilde \pi_{\mathfrak g}: \mathfrak g[1] \times E \to \mathfrak g [1]$. Composition on the right with the pull-back $\tilde \pi{}_{\mathfrak g}^\ast$
establishes a projection:
\[
\mathfrak X_{\mathrm{lin}}( \mathfrak g[1] \times E)^\bullet \to \mathfrak X_{\mathrm{rel}} (\pi_{\mathfrak g})^\bullet, \quad X \mapsto X \circ \tilde \pi{}^\ast_{\mathfrak g}
\]
(beware, the range consists of $\pi_{\mathfrak g}$-relative, not $\tilde \pi_{\mathfrak g}$-relative, vector fields) whose kernel consists of $\tilde \pi_\mathfrak g$-vertical linear vector fields $\mathfrak X^{\pi_\mathfrak g}_{\mathrm{lin}} (\mathfrak g[1] \times E)^\bullet$. Hence there is a short exact sequence of cochain complexes:
\begin{equation}\label{SES:gE}
0 \to \mathfrak X^{\tilde \pi_\mathfrak g}_{\mathrm{lin}} (\mathfrak g[1] \times E)^\bullet \to \mathfrak X_{\mathrm{lin}} (\mathfrak g[1] \times E)^\bullet \to \mathfrak X_{\mathrm{rel}} (\pi_{\mathfrak g})^\bullet \to 0.
\end{equation}
Using the projection $\mathfrak X_{\mathrm{lin}} (\mathfrak g[1] \times E)^\bullet \to \mathfrak X (\mathfrak g[1] \times M)^\bullet$, we can combine sequences (\ref{SES:gE}) and (\ref{SES:gM}) in an exact diagram
\[
\begin{array}{c}
\xymatrix@C=15pt@R=15pt{
& 0 \ar[d] & 0 \ar[d] & & \\
0 \ar[r] & \operatorname{\mathfrak{End}}(\mathfrak g [1] \times E)^\bullet  \ar[d] \ar@{=}[r] & \operatorname{\mathfrak{End}} (\mathfrak g [1] \times E)^\bullet \ar[r] \ar[d] & 0 \ar[d] \\
0 \ar[r] & \mathfrak X^{\tilde \pi_\mathfrak g}_{\mathrm{lin}} (\mathfrak g[1] \times E)^\bullet \ar[r] \ar[d] & \mathfrak X_{\mathrm{lin}} (\mathfrak g[1] \times E)^\bullet \ar[r] \ar[d] & \mathfrak X_{\mathrm{rel}} (\pi_{\mathfrak g})^\bullet \ar[r] \ar@{=}[d] & 0 \\
0 \ar[r] & \mathfrak X^{\pi_\mathfrak g} (\mathfrak g[1] \times M)^\bullet \ar[r] \ar[d] & \mathfrak X(\mathfrak g[1] \times M)^\bullet \ar[r] \ar[d] & \mathfrak X_{\mathrm{rel}} (\pi_{\mathfrak g})^\bullet \ar[r] \ar[d] & 0 \\
 & 0 & 0 & 0 &
}
\end{array},
\]
where, as usual, $\operatorname{\mathfrak{End}}(\mathfrak g [1] \times E)^\bullet $ consists of graded endomorphism of the vector bundle $\mathfrak g[1] \times E \to \mathfrak g[1] \times M$ (covering the identity). Now, $\mathfrak X_{\mathrm{lin}}(\mathfrak g[1] \times E)^\bullet$ is the linear deformation complex of $(\mathfrak g \ltimes E \Rightarrow E ; \mathfrak g \ltimes M \Rightarrow M)$, and $\operatorname{\mathfrak{End}} (\mathfrak g [1] \times E)^\bullet$ is canonically isomorphic to the Chevalley-Eilenberg cochain complex of $\mathfrak g$ with coefficients in $\operatorname{\mathfrak{End}} E$, endomorphisms of $E$ (covering the identity). Finally, $\mathfrak X^{\tilde \pi_\mathfrak g}_{\mathrm{lin}} (\mathfrak g[1] \times E)^\bullet$ is canonically isomorphic to the Chevalley-Eilenberg cochain complex of $\mathfrak g$ with coefficients in $\mathfrak D (E)$. The isomorphism 
\[
C^\bullet (\mathfrak g , \mathfrak D (E)) \overset{\cong}{\longrightarrow} \mathfrak X^{\tilde \pi_\mathfrak g}_{\mathrm{lin}} (\mathfrak g[1] \times E)^\bullet
\]
maps a cochain $\omega \otimes \Delta$ to the vector field $\tilde \pi{}^\ast_{\mathfrak g}(f_{\omega}) X_{\Delta}$, where $f_\omega$ is the function on $\mathfrak g [1]$ corresponding to $\omega \in C^\bullet (\mathfrak g)$, and $X_\Delta$ is the unique $\tilde \pi_{\mathfrak g}$-vertical vector field on $\mathfrak g [1] \times E$ projecting on the linear vector field on $E$ corresponding to derivation $\Delta$.

We conclude that there is an exact diagram of cochain complexes
\[
\begin{array}{c}
\xymatrix@C=15pt@R=15pt{
& 0 \ar[d] & 0 \ar[d] & & \\
0 \ar[r] & C^\bullet (\mathfrak g, \operatorname{\mathfrak{End}} E) \ar[d] \ar@{=}[r] & C^\bullet (\mathfrak g, \operatorname{\mathfrak{End}} E) \ar[r] \ar[d] & 0 \ar[d] \\
0 \ar[r] & C^\bullet (\mathfrak g, \mathfrak{D} (E)) \ar[r] \ar[d] & C^\bullet_{\mathrm{def}, \mathrm{lin}} (\mathfrak g\ltimes E) \ar[r] \ar[d] & C^{\bullet +1} (\mathfrak g, \mathfrak g_M) \ar[r] \ar@{=}[d] & 0 \\
0 \ar[r] & C^\bullet (\mathfrak g, \mathfrak X(M)) \ar[r] \ar[d] & C^\bullet_{\mathrm{def}}(\mathfrak g \ltimes M) \ar[r] \ar[d] & C^{\bullet +1} (\mathfrak g, \mathfrak g_M) \ar[r] \ar[d] & 0 \\
 & 0 & 0 & 0 &
}
\end{array}.
\]
This proves the following
\begin{proposition}
Let $\mathfrak g$ be a Lie algebra acting on a vector bundle $E \to M$ by infinitesimal vector bundle automorphisms. The linear deformation cohomology of the VB-algebroid $(\mathfrak g \ltimes E \Rightarrow E, \mathfrak g \ltimes M \Rightarrow M)$ fits in the exact diagram:
\[
\begin{array}{c}
\xymatrix@C=15pt@R=15pt{
& \vdots \ar[d] &  \vdots \ar[d] &  \vdots \ar[d] &  \vdots \ar[d] & \\
\cdots \ar[r] & H^k (\mathfrak g, \operatorname{\mathfrak{End}} E) \ar@{=}[r] \ar[d]& H^k(\mathfrak g, \operatorname{\mathfrak{End}} E) \ar[r] \ar[d]& 0 \ar[r] \ar[d] &  H^{k+1}(\mathfrak g, \operatorname{\mathfrak{End}} E) \ar[r] \ar[d]& \cdots \\
\cdots \ar[r] & H^k (\mathfrak g, \mathfrak D (E)) \ar[r] \ar[d]& H^k_{\mathrm{def}, \mathrm{lin}}(\mathfrak g \ltimes E) \ar[r] \ar[d]& H^{k +1} (\mathfrak g, \mathfrak g_M) \ar[r] \ar@{=}[d]&  H^{k+1}(\mathfrak g, \mathfrak D (E)) \ar[r] \ar[d]& \cdots \\
\cdots \ar[r] & H^k (\mathfrak g, \mathfrak X (M)) \ar[r] \ar[d] & H^k_{\mathrm{def}}(\mathfrak g \ltimes M) \ar[r] \ar[d]& H^{k +1}(\mathfrak g, \mathfrak g_M) \ar[r] \ar[d]&  H^{k+1}(\mathfrak g, \mathfrak X (M)) \ar[r] \ar[d] & \cdots \\
\cdots \ar[r] & H^{k+1} (\mathfrak g, \operatorname{\mathfrak{End}} E) \ar@{=}[r] \ar[d]& H^{k+1}(\mathfrak g, \operatorname{\mathfrak{End}} E) \ar[r] \ar[d]& 0 \ar[r] \ar[d] &  H^{k+2}(\mathfrak g, \operatorname{\mathfrak{End}} E) \ar[r] \ar[d]& \cdots \\
& \vdots  &  \vdots  &  \vdots  &  \vdots &
}
\end{array}.
\]
\end{proposition}

\subsection{Type 1 VB-algebroids}\label{sec:type_1}

Let $(W \Rightarrow E; A \Rightarrow M)$ be a VB-algebroid with core $C$. The \emph{core-anchor} of $(W \Rightarrow E; A \Rightarrow M)$ is the vector bundle map $\partial : C \to E$ defined as follows. Let $\chi$ be a section of $C$, and let $\overline \chi$ be the corresponding core section of $W \to E$. The anchor $\rho : W \to TE$ maps $\overline \chi$ to a core vector field $\rho (\overline \chi)$ on $E$. In its turn $\rho (\overline \chi)$ is the vertical lift of a section $\varepsilon$ of $E$. By definition, $\partial \chi = \varepsilon$. 

According to a definition by Gracia-Saz and Mehta \cite{gracia-saz:vb}, a VB-algebroid is \emph{type $1$} (resp.~\emph{type $0$}) if the core-anchor is an isomorphism (resp.~is the zero map). More generally, $(W \Rightarrow E; A \Rightarrow M)$ is \emph{regular} if the core-anchor has constant rank. In this case $(W \Rightarrow E; A \Rightarrow M)$ is the direct sum of a \emph{type $1$} and a \emph{type $0$} VB-algebroid, up to isomorphisms. So type 1 and type 0 VB-algebroids are the building blocks of regular VB-algebroids. In this subsection we discuss linear deformation cohomologies of type 1 VB-algebroids.

Let $(W \Rightarrow E; A \Rightarrow M)$ be a type 1 VB-algebroid, and denote by $q : E \to M$ the projection. Gracia-Saz and Mehta show that $(W \Rightarrow E; A \Rightarrow M)$ is canonically isomorphic to the VB-algebroid $(q^! A \Rightarrow E; A \Rightarrow M)$ \cite{gracia-saz:vb}. Here $q^! A \Rightarrow E$ is the \emph{pull-back Lie algebroid}. Recall that its total space $q^! A$ is the fibered product $
q ^{!} A := TE \, {{}_{dq } \times _{\rho }}\, A
$. Hence, sections of $q^! A \to E$ are pairs $(X, \alpha )$, where
$X$ is a vector field on $E$ and $\alpha$ is a section of the pull-back
bundle $q^{\ast }A \to E$, with the additional property that $d q (X_{e}) = \rho (
\alpha_{q(e)})$ for all $e \in E$. Then there exists a
unique Lie algebroid structure $q^{!} A \Rightarrow E$ such that the
anchor $q^{!} A \to TE$ is the projection $(X, \alpha) \mapsto X$,
and the Lie bracket is given by
\[
\left [  (X, q^{\ast }\alpha ), (Y, q^{\ast }\beta) \right ]  =
\left (  [X,Y], q^{\ast }[\alpha, \beta]\right ),
\]
on sections of the special form $(X, q^{\ast }\alpha ), (X, q^{\ast}
\beta )$, with $\alpha , \beta \in \Gamma (A)$. Finally, $(q^! A \Rightarrow E; A \Rightarrow M)$
is a VB-algebroid, and every VB-algebroid of type $1$ arises in this way
(up to isomorphisms).

As $E \to M$ is a vector bundle, it has contractible fibers. So, according to \cite{spar:def}, $q^! A \Rightarrow E$ and $A \Rightarrow M$ share the same deformation cohomology . As an immediate consequence we get that the canonical map $C^\bullet_{\mathrm{def}, \mathrm{lin}} (q^! A) \to C^\bullet_{\mathrm{def}} (A)$ induces an injection in cohomology. We want to show that, even more, it is a quasi-isomorphism. To do this it is enough to prove that the kernel $\operatorname{\mathfrak{End}} (q^! A [1]_E)^\bullet$ of $C^\bullet_{\mathrm{def}, \mathrm{lin}} (q^! A) \to C^\bullet_{\mathrm{def}} (A)$ is acyclic. We use graded geometry again. So, consider the pull-back diagram
\[
\begin{array}{c}
\xymatrix{q^! A \ar[r] \ar[d]&  TE \ar[d]^-{dq} \\
A \ar[r]^-{\rho} & TM
} 
\end{array}.
\]
All vertices are vector bundles, and shifting by one the degree in their fibers, we get a pull-back diagram of DG-manifolds:
\[
\begin{array}{c}
\xymatrix{q^! A[1]_E \ar[r] \ar[d]_-{\tilde q}&  T[1]E \ar[d]^-{dq} \\
A[1] \ar[r]^-{\rho} & T[1]M
} 
\end{array}.
\]
This shows, among other things, that there is a canonical isomorphism of DG-modules:
\[
\operatorname{\mathfrak{End}} (q^! A[1]_E)^\bullet = C^\bullet (A) \otimes_{\Omega^\bullet (M)} \operatorname{\mathfrak{End}}( T[1] E)^\bullet.
\]
From Proposition \ref{prop:TE}, exact sequence (\ref{eq:SESDGLAs}), and the fact that the
 deformation cohomology of $TM \Rightarrow M$ is trivial, $\operatorname{\mathfrak{End}}
 (T[1] E)^\bullet $ is acyclic. Actually there is a canonical contracting homotopy $h' : 
 \operatorname{\mathfrak{End}}( T[1] E)^\bullet \to \operatorname{\mathfrak{End}}( T[1] 
 E)^{\bullet -1}$. Indeed, there is a canonical contracting homotopy $H : \mathfrak X (T[1]
  E)^\bullet \to \mathfrak X (T[1] E)^{\bullet -1}$ restricting to both $\mathfrak X_{\mathrm{lin}}
   (T[1] E)^\bullet$ and $ \operatorname{\mathfrak{End}}( T[1] E)^\bullet$ (see, 
   e.g.~\cite{vit:sh-Lie, spar:def}, for a definition of $H$). Then, $h'$ is simply the restriction of $H$, 
   and it is graded $C^\bullet(A)$-linear. Finally, we define a contracting homotopy $h : 
   \operatorname{\mathfrak{End}} (q^! A[1]_E)^\bullet \to \operatorname{\mathfrak{End}}(q^! 
   A[1]_E)^{\bullet -1} $ by putting $h (\omega \otimes \Phi) := (-)^\omega \omega \otimes 
   h' (\Phi)$, for all $\omega \in C^\bullet (A)$, and all $\Phi \in \operatorname{\mathfrak{End}}
   (T[1] E)^\bullet $. Summarizing, we proved the following
\begin{proposition}
Let $(W \Rightarrow E; A \Rightarrow M)$ be a type 1 VB-algebroid. Then the canonical surjection $C^\bullet_{\mathrm{def}, \mathrm{lin}}(W) \to C^\bullet_{\mathrm{def}}(A)$ is a quasi-isomorphism. In particular,  $H^\bullet_{\mathrm{def}, \mathrm{lin}}(W) = H^\bullet_{\mathrm{def}}(A)$.
\end{proposition}
Morally, deforming a type 1 VB-algebroid is the same as deforming its base Lie algebroid.

\noindent \textbf{Acknowledgments.} We thank the referee for reading carefully our manuscript, and for several suggestions that improved a lot the presentation. PPLP and LV are members of GNSAGA of INdAM.

\appendix

\section{The homogeneity structure of a vector bundle a linear multivectors}\label{Sec:homogeneity}

In this appendix, for the reader's convenience, we recall the well-known concepts of homogeneity structure of a vector bundle and of linear multivectors 
on its total space. We make no claim of originality: these ideas appeared (probably for the first time, e.g.) in \cite{grab:vect}, \cite{iglesias:univ} and \cite{bursztyn:mult}, respectively. In the last two references, the reader can also find the proof of (a version of) Propositions 
\ref{prop:lin_multi}. With respect to those references, we will offer just a slightly different point of view, in order to make the presentation consistent. Notations and conventions in this appendix are used throughout the paper, sometimes without further comments.

Let $E \to M$ be a vector bundle. The monoid $\mathbb R_{\geq 0}$ of non-negative real numbers acts on $E$ by homotheties $h_\lambda: E \to E$ (fiber-wise scalar multiplication). The action $h : \mathbb R_{\geq 0} \times E \to E$, $e \mapsto h_\lambda (e)$, is called the \emph{homogeneity structure} of $E$. The homogeneity structure (together with the smooth structure) fully characterizes the vector bundle structure \cite{grab:vect}. In particular, it determines the addition. This implies that \emph{every notion that involves the linear structure of $E$ can be expressed in terms of $h$ only}: for example, a smooth map between the total spaces of two vector bundles is a bundle map if and only if it commutes with the homogeneity structures.

The homogeneity structure isolates a distinguished subspace in the algebra 
$\mathfrak X^\bullet (E)$ of multivectors on the total space $E$ of the vector bundle.

\begin{definition} 

A multivector $X \in \mathfrak X^\bullet (E)$ is (\emph{homogeneous}) \emph{of weight $q$} if and only if
\begin{equation}\label{eq:hom_multivect}
h_\lambda^* X = \lambda^q X
\end{equation}
for all $\lambda > 0$.  The space of $k$-vector fields of weight $q$ on $E$ will be denoted $\mathfrak X^k_q (E)$. We denote simply by $C^\infty_q(E) := \mathfrak X^0_q (E)$ the space of functions of weight $q$ and by $\mathfrak X_q (E) := \mathfrak X^1_q(E)$ the space of vector fields of weight $q$.
\end{definition}

Clearly, for $q \geq 0$, weight $q$ functions coincide with functions on $E$ that are fiber-wise polynomial of degree $q$, while for $q < 0$ there are no non-zero functions of weight $q$. In particular, weight-zero functions are fiber-wise constant functions, i.e. pull-backs of functions on the base $M$. We refer to them as \emph{core functions} and we denote $C^\infty_{\mathrm{core}}(E) := C^\infty_0(E)$. 

The functorial properties of the pull-back imply that the grading defined by the weight is natural with respect to all the usual operations on functions
and (multi)vector fields. From this remark, we easily see that \emph{there are no non-zero $k$-vector fields of weight less than $-k$}.

\begin{definition}
A functions 
on $E$ is \emph{linear} if it is of weight 1. More generally, a $k$-vector field is \emph{linear} if it is of weight $1-k$. We denote by $C^\infty_{\mathrm{lin}}(E)$, 
$\mathfrak X_{\mathrm{lin}}(E)$ and $\mathfrak X^\bullet_{\mathrm{lin}}(E)$ 
the spaces of linear functions, 
vector fields and multivectors, 
respectively. 
\end{definition}

Linear functions are precisely fiber-wise linear functions. The definition of linear multivector may sound a little strange, but it is motivated (among other things) by the following proposition.

\begin{proposition}\label{prop:lin_multi}
Let $X \in \mathfrak X^k (E)$. The following conditions are equivalent:
\begin{enumerate} 
\item $X$ is linear;
\item $X$ takes
	\begin{enumerate}
	\item $k$ linear functions to a linear function,
	\item $k-1$ linear functions and a core function to a core function,
	\item $k-i$ linear functions and $i$ core functions to $0$, for every $i \geq 2$;
	\end{enumerate}
\item If $(x^i)$ are local coordinates on $M$ and $(u^\alpha)$ are linear fiber coordinates on $E$, $X$ is locally of the form
\begin{equation}\label{eq:mult_lin_coord}
X = X^{\alpha_1 \dots \alpha_{k-1} i} (x) \dfrac{\partial}{\partial u^{\alpha_1}} \wedge \dots \wedge \dfrac{\partial}{\partial u^{\alpha_{k-1}}} \wedge \dfrac{\partial}{\partial x^i} + X^{\alpha_1 \dots \alpha_k}_\beta (x) u^\beta \dfrac{\partial}{\partial u^{\alpha_1}} \wedge \dots \wedge \dfrac{\partial}{\partial u^{\alpha_k}}.
\end{equation}
\end{enumerate}
\end{proposition}


\begin{thebibliography}{99}
 
\bibitem{abad:ruth} C. Arias Abad and M. Crainic, Representations up to homotopy of Lie algebroids, \emph{J. reine angew. Math.} \textbf{663} (2012), 91--126.


\bibitem{behrend:diff} K. Behrend and P. Xu, Differentiable stacks and gerbes, \emph{J. Symplect. Geom.} \textbf{9.3} (2011), 285--341.



\bibitem{bursztyn:mult} H. Bursztyn and A. Cabrera, Multiplicative forms at the infinitesimal level, \emph{Math. Ann.} \textbf{353.3} (2012), 663--705.


\bibitem{cabrera:hom} A. Cabrera and T. Drummond, Van Est isomorphism for homogeneous cochains, \emph{Pacific J. Math.} \textbf{287} (2017), 297--336.

\bibitem{crainic:def} M. Crainic and I. Moerdijk, Deformations of Lie brackets: cohomological aspects, \emph{J. Eur. Math. Soc.}, \textbf{287} (2008), 1037--1059.

\bibitem{crainic:def2} M. Crainic and J. N. Mestre and I. Struchiner, Deformations of Lie groupoids (2015); e-print: \texttt{arXiv:1510.02530}.

\bibitem{delhoyo:morita} M. del Hoyo and C. Ortiz, Morita equivalences of vector bundles (2016); e-print: \texttt{arXiv:1611.06896}.

\bibitem{etv:infinitesimal} C. Esposito, A. G. Tortorella and L, Vitagliano, Infinitesimal automorphisms of VB-groupoids and algebroids (2016); e-print: \texttt{arXiv:1611.06896}.

\bibitem{grab:Lie} J. Grabowski and K. Grabowska, Lie brackets on affine bundles, \emph{Ann. Global Anal. Geom.} \textbf{24} (2003), 101--130.

\bibitem{grab:vect} J. Grabowski and M. Rotkiewicz, Higher vector bundles and multi-graded symplectic manifolds, \emph{J. Geom. Phys.} \textbf{59} (2009), 1285--1305.

\bibitem{gracia-saz:vb} A. Gracia-Saz and R. A. Mehta, Lie algebroid structures on double vector bundles and representation theory of Lie algebroids, \emph{Adv. Math.} \textbf{223}, (2010), 1236--1275.

\bibitem{iglesias:univ} D. Iglesias-Ponte, C. Laurent-Gengoux and P. Xu, Universal lifting theorem and quasi-Poisson groupoids, \emph{J. Eur. Math. Soc.}, \textbf{14.3} (2012), 681--731.


\bibitem{lapastina:def2} P.P. La Pastina and L. Vitagliano, Deformations of vector bundles over Lie groupoids, in preparation.

\bibitem{mackenzie:double} K. C. H. Mackenzie, Double Lie algebroids and the double of a Lie bialgebroid (1998); e-print: \texttt{arXiv:math/9808081}.

\bibitem{mackenzie:ehresmann} K. C. H. Mackenzie, Drinfel'd doubles and Ehresmann doubles for Lie algebroids and Lie bialgebroids, \emph{Electronic Research Announcements of the American Mathematical Society} \textbf{4} (1998), 74--87.

\bibitem{mackenzie} K. C. H. Mackenzie, General theory of Lie groupoids and algebroids, (2005), Cambridge Univ.~Press, Cambridge, 2005.


\bibitem{mehta:thesis} R. A. Mehta, Supergroupoids, double structures, and equivariant cohomology, \emph{Ph.D. thesis}, University of California, Berkeley, 2006.

\bibitem{nijenhuis:def_lie} A. Nijenhuis and R. W. Richardson, Deformations of Lie algebra structures, \emph{J. Math. Mech.} \textbf{17} (1967), 89--105.

\bibitem{nijenhuis:coh} A. Nijenhuis and R. W. Richardson, Cohomology and deformations in graded Lie algebras, \emph{Bull. Amer. Math. Soc.} \textbf{72} (1967), 1--29.

\bibitem{nijenhuis:def} A. Nijenhuis and R. W. Richardson, Deformations of homomorphisms of Lie groups and Lie algebras, \emph{Bull. Amer. Math. Soc.} \textbf{73} (1967), 175--179.

\bibitem{pradines:theorie} J. Pradines, Th\'eorie de Lie pour les groupo\"ides diff\'erentiables, \emph{C. R. Acad. Sci. Paris} \textbf{264} (1967), 245--248.

\bibitem{spar:def} G. Sparano and L. Vitagliano, Deformation cohomology of Lie algebroids and Morita equivalence, \emph{C. R. Acad. Sci. Paris, Ser. I} \textbf{356} (2018), 376--381.

\bibitem{vaintrob:lie} A. Yu. Vaintrob, Lie algebroids and homological vector fields, \emph{Uspekhi Mat. Nauk} \textbf{52} (1997), 161--162.

\bibitem{vit:sh-Lie} L. Vitagliano, On the strong homotopy Lie-Rinehart algebra of a foliation, \emph{Commun. Contemp. Math.} \textbf{16} (2014), 1450007 (49 pages).

\bibitem{Vit:vv-forms} L. Vitagliano, Vector bundle valued differential forms on $\mathbb N Q$-manifolds, \emph{Pacific J. Math.} \textbf{283} (2016), 449--482.

\bibitem{voronov:q} T. Voronov, Q-manifolds and Mackenzie theory, \emph{Comm. Math. Phys.} \textbf{315} (2012), 279--310.

\end{thebibliography}
\end{document}